\documentclass{CMSLATEX}
\usepackage{latexsym, amssymb, enumerate, amsmath}
\usepackage{enumerate}

\usepackage{graphicx}
\renewcommand{\theequation}{\arabic{section}.\arabic{equation}}

\newtheorem{thm}{Theorem}[section]

\newtheorem{lem}[thm]{Lemma}

\newtheorem{defn}[thm]{Definition}

\newtheorem{example}[thm]{Example}

\newtheorem{rem}[thm]{Remark}
\begin{document}

\title{Fourth order quasi-compact difference schemes for (tempered) space fractional diffusion equations\thanks{This work was supported by the National Natural Science Foundation of China under Grant  Nos. 11271173 and 11471150.}
}
          %For each author, make a block with the following four macros:
\author{Yanyan Yu\thanks{School of Mathematics and Statistics, Gansu Key Laboratory of Applied Mathematics and Complex Systems, Lanzhou University, Lanzhou 730000, People's Republic of China  (yuyy1029@126.com).}
  \and   Weihua Deng\thanks {School of Mathematics and Statistics, Gansu Key Laboratory of Applied Mathematics and Complex Systems, Lanzhou University, Lanzhou 730000, People's Republic of China (dengwh@lzu.edu.cn).} \and
   Yujiang Wu\thanks {School of Mathematics and Statistics, Gansu Key Laboratory of Applied Mathematics and Complex Systems, Lanzhou University, Lanzhou 730000, People's Republic of China (myjaw@lzu.edu.cn).} }
          %{Put the URL for your home page here if you have one}

\pagestyle{myheadings} \markboth{Fourth order quasi-compact difference schemes}{Yanyan Yu, Weihua Deng  and Yujiang Wu}
\maketitle
\begin{abstract}
The continuous time random walk (CTRW) underlies many fundamental processes in non-equilibrium statistical physics. When the jump length of CTRW obeys a power-law distribution, its corresponding Fokker-Planck equation has space fractional derivative, which characterizes L\'{e}vy flights. Sometimes the infinite variance of L\'{e}vy flight discourages it as a physical approach; exponentially tempering the power-law jump length of CTRW makes it more `physical' and the tempered space fractional diffusion equation appears. This paper provides the basic strategy of deriving the high order quasi-compact discretizations for space fractional derivative and tempered space fractional derivative. The fourth order quasi-compact discretization for space fractional derivative is applied to solve space fractional diffusion equation and the unconditional stability and convergence of the scheme are theoretically proved and numerically verified. Furthermore, the tempered space fractional diffusion equation is effectively solved by its counterpart of the fourth order quasi-compact scheme; and the convergence orders are verified numerically.
\end{abstract}

\begin{keywords}
space fractional derivative,  tempered space fractional derivative,  shifted Gr\"{u}nwald discretization,  quasi-compact difference scheme, numerical  stability and convergence.

\smallskip

{\bf Subject classifications.  65M06, 65M12, 26A33}
\end{keywords}

%Subject Classifications.

\section{Introduction}
In recent years, more and more scientific and engineering problems are involved in fractional calculus. They range from relaxation oscillation phenomena \cite{mainardi1996fractional}  to viscoelasticity \cite{bagley1983theoretical}, and from control theory
\cite{vinagre2000some} to transport problem \cite{metzler2004restaurant}. The fractional diffusion equation has been put forward as a more suitable model to describe ion channel gating dynamics \cite{goychuk2004fractional} and subdiffusive anomalous transport in an external field \cite{barkai2000continuous}, which are resulted in from the continuous time random walk (CTRW) in the scaling limit. The CTRW is a mathematical formalization of a path that consists of a succession of random steps including the elements of random waiting time and jump length; and it underlies many fundamental stochastic processes in statistical physics. When the first moment of the distribution of waiting time and the second moment of jump length are finite, the probability density function (PDF) of the particle's location and time satisfies the classical diffusion equation. However, if the jump length obeys the power-law distribution, the PDF of the particle's location and time is the solution of space fractional diffusion equation; and the corresponding dynamics is called L\'{e}vy flight. Sometimes the jumps of the particles are limited by the finite size of the physical system and the infinite variance of L\'{e}vy flight discourages it as a physical approach. So the power-law distribution of the jump length is expected to be truncated \cite{mantegna1994stochastic} or exponentially tempered \cite{Cartea}. For the CTRW with the distribution of the tempered jump length $|x|^{-(1+\alpha)}e^{-\lambda |x|}$ \cite{chakrabarty2011tempered}, the corresponding PDF of the particles satisfies the tempered space fractional diffusion equation  \cite{Cartea}.

It seems that there are less works for the numerical solutions of tempered space fractional diffusion equation \cite{li2014high}.  However, for the space fractional diffusion or advection-diffusion equation, much progress has been made for its numerical methods, e.g., \cite{liu2004numerical, meerschaert2004finite, meerschaert2006finite, meerschaert2006, Tadjeran2006, tadjeran2007second, nasir2013second, Tian2012, zhou2013quasi, deng2013efficient}. Transforming the Riemann-Liouville fractional derivative to Caputo fractional derivative, the space fractional Fokker-Planck equation is solved by the method of lines in \cite{liu2004numerical}. Using the superconvergence of Gr\"{u}nwald discretization at a particular point, a second order finite difference scheme is proposed in \cite{nasir2013second}. Based on the difference discretization and spline approximation to the Riemann-Liouville fractional derivative, a second order scheme is presented for the three dimensional space fractional partial differential equations in \cite{deng2013efficient}. Currently, the most popular discretization scheme for the space Riemann-Liouville fractional derivative seems to be the weighted and shifted Gr\"{u}nwald (WSGD) operator. The first order WSGD operator is firstly presented and detailedly discussed in \cite{meerschaert2004finite, meerschaert2006finite, meerschaert2006} and the second order convergence is obtained by using extrapolation method \cite{Tadjeran2006, tadjeran2007second}. The second order WSGD operator is given in \cite{Tian2012}; and the third order compact WSGD (CWSGD) is presented in \cite{zhou2013quasi}. Following the idea of weighting and shifting Gr\"{u}nwald operator, this paper provides the basic strategy of deriving the quasi-compact scheme with any desired convergence orders for space fractional diffusion equation; and it can also be extended to solve the tempered space fractional diffusion equation. The fourth order quasi-compact scheme is detailedly discussed in solving space fractional diffusion equation, including stability and convergence analysis and numerical verification of convergence orders. The fourth order quasi-compact scheme for tempered space fractional diffusion equation is also proposed and effectively used to solve the equation; and the convergence orders are numerically verified.

The outline of this paper is as follows. In Sec. 2, the high order quasi-compact discretizations are presented to approximate space Riemann-Liouville fractional derivative. In Sec. 3, following the obtained quasi-compact discretizations, the high order quasi-compact scheme for the one dimensional space fractional diffusion equation is designed and its stability and convergence analysis are performed. Sec. 4 focuses on the quasi-compact scheme and the corresponding stability and convergence analysis in two dimensional case. The high order quasi-compact discretizations is extended to tempered space fractional derivative in Sec. 5 and the corresponding scheme is derived to solve tempered space fractional diffusion equation. In Sec. 6, numerical experiments are performed to testify the efficiency and verify the convergence orders of the schemes. We conclude the paper with some discussions in the last section.

%-------------------------------------------------------------------
\section{Quasi-compact discretizations for Riemann-Liouville space fractional derivatives}
We first introduce some definitions and lemmas, including  Riemann-Liouville fractional derivatives and  shifted Gr\"{u}nwald-Letnikov  formulations.
\begin{defn}\label{drl}\cite{I1998}
If the function $u(x)$ is defined in the interval $(a,b)$ and regular enough, then the $\alpha$-th order left and right Riemann-Liouville fractional derivatives are, respectively, defined as
\begin{equation}\label{lefrie}
_a D_x^{\alpha}u(x)=\frac{1}{\Gamma(n-\alpha)}\frac{d^n}{dx^n} \int^{x}_{a}(x-s)^{n-\alpha-1}u(s) ds ,\quad n-1<\alpha<n
\end{equation}
and
\begin{equation}
_xD_b^{\alpha}u(x)=\frac{(-1)^n}{\Gamma(n-\alpha)}\frac{d^n}{dx^n} \int^{b}_{x}(s-x)^{n-\alpha-1}u(s)ds ,\quad n-1<\alpha<n,
\end{equation}
where $a$ can be $-\infty$ and $b$ can be  $+\infty$.
\end{defn}
%Let the order $\alpha \in R$ and the function $u(x)\in C(R)$,

And the standard left and right Gr\"{u}nwald-Letnikov  formulations which can be potentially used to approximate the left and right  Riemann-Liouville fractional derivatives are, respectively, given as
\begin{equation}%\label{gl}
_{a}D^\alpha_xu(x)=\lim\limits_{h\rightarrow 0} \frac{1}{h^\alpha} \sum\limits_{k=0}^{[\frac{x-a}{h}]} g_k^{(\alpha)}u(x-kh)
\end{equation}
and
\begin{equation}
_xD^\alpha_{b}u(x)=\lim\limits_{h\rightarrow 0} \frac{1}{h^\alpha} \sum\limits_{k=0}^{[\frac{b-x}{h}]} g_k^{(\alpha)}u(x+kh),
\end{equation}
where the Gr\"{u}nwald weights $g_k^{(\alpha)}=\frac{\Gamma(k-\alpha)}{\Gamma(-\alpha)\Gamma(k+1)}$ are the coefficients of the power series expansion of $(1-z)^\alpha$. %And they satisfy the following property.
%\begin{lem}(\cite{Tadjeran2006})
%The coefficients in (\ref{gl}) satisfy the following properties for $1<\alpha<2$,
%\begin{enumerate}
%\item $g_0^{(\alpha)}=1,g_1^{(\alpha)}=-\alpha<0$,
%\item $0\leq\cdots\leq g_3^{(\alpha)} \leq g_2^{(\alpha)}\leq 1$,
%\item   $\sum\limits_{k=0}^{\infty}g_k^{(\alpha)}=0,\sum\limits_{k=0}^{m}g_k^{(\alpha)}<0,\,m\geq1. $
%\end{enumerate}
%\end{lem}
For getting the stable scheme, a shifted Gr\"{u}nwald-Letnikov operator
is proposed to approximate the left Riemann-Liouville  fractional derivative with first order accuracy  \cite{Tadjeran2006}.
%And in []

\begin{lem}[\cite{Tadjeran2006}] \label{lem1}%----------------------------------------------------------------------------------
 Let $1<\alpha<2$, $u\in C^{n+3}(R)$, and $D^ku(x)\in L^1(R)$, $k=0,1,  \cdots,n+3$. For any integer $p$, define the left shifted  Gr\"{u}nwald-Letnikov operator by
\begin{equation}\label{delf}
\Delta_p^\alpha u(x):=\frac{1}{h^\alpha}\sum\limits_{k=0}^\infty g_k^{(\alpha)}u(x-(k-p)h).
\end{equation}
Then we have
\begin{equation}\label{scheme0}
\Delta_p^\alpha u(x)={_{-\infty}D}^\alpha_x u(x)+\sum\limits_{l=1}^{n-1}a_{p, l}^\alpha \, _{-\infty}D_x^{\alpha+l}u(x)h^l+O(h^n),
\end{equation}
uniformly in $x\in R$, where the weights $a_{p, l}^\alpha $ are the coefficients of the power series expansion of the function $(\frac{1-e^{-z}}{z})^\alpha e^{pz}$, and the first four terms of the coefficients are
$a_{p, 0}^\alpha =1$, $a_{p, 1}^\alpha =p-\alpha/2$, $a_{p, 2}^\alpha =(\alpha+3\alpha^2-12\alpha p+12p^2)/24$, and $a_{p, 3}^\alpha= (8 p^3 + 2 p \alpha - 12 p^2 \alpha - \alpha^2 +
   6 p\alpha^2 - \alpha^3)/48$.

\end{lem}%-------------------------------------------------------------------------------------
To approximate the right Riemann-Liouville fractional derivative $   _x D ^\alpha_{ \infty} u(x)$, the right shifted Gr\"{u}nwald-Letnikov operator is given by
$
\Lambda_p^\alpha f(x):=\frac{1}{h^\alpha}\sum\limits_{k=0}^\infty g_k^{(\alpha)}f(x+(k-p)h) .
$
In the finite interval  $[a,b]$, the shifted Gr\"{u}nwald-Letnikov fractional derivatives are
\begin{equation} \label{eq7}
\tilde{\Delta}_p^\alpha u(x)=\frac{1}{h^\alpha}\sum\limits_{k=0}^{[\frac{x-a }{h}]+p} g_k^{(\alpha)}u(x-(k-p)h)
\end{equation}
and
\begin{equation}\label{eq8}
\tilde{\Lambda}_p^\alpha u(x)=\frac{1}{h^\alpha}\sum\limits_{k=0}^{[\frac{b -x}{h}]+p} g_k^{(\alpha)}u(x+(k-p)h).
\end{equation}
In the remaining analysis of the paper, for a function defined in the bounded interval, we suppose that it has been zero extended to $R$ whenever the value of $u(x)$ outside of the bounded interval is used.
%If the function $u(x)$  supported in the finite interval $[a,b]$ vanish in $R\backslash[a,b]$, then it can be extended to   $R$.
%In this case, the left (right) boundary  in left (right) Riemann-Liouville fractional  operator   can be
%replaced by $-\infty$ (or $+\infty$).

%
\subsection{ Fourth order quasi-compact approximation to the Riemann-Liouville fractional derivative}%**********************************************************************************************
According to the definitions of the shifted Gr\"{u}nwald-Letnikov fractional derivatives, we know that $p$ can be any integer. In order to ensure that the nodes in (\ref{eq7}) or (\ref{eq8}) are within the bounded interval,  % The integer $p$ in (\ref{delf}) is the number of the points $x$ ,
we need to choose the integer $p\in\{1, 0, -1\}$ when approximating non-periodic fractional differential equation in the
bounded interval.
%And the corresponding results can be obtained for the right Riemann-Liouville  fractional derivative.  In this section we choose three different integers.
% and  denote them by $r, p$  and $q$.
 Inspired by the shifted Gr\"{u}nwald-Letnikov operator and the Taylor expansion, we derive the following fourth order combined quasi-compact approximations.
\begin{thm}\label{th2.1}%-----------------------------------------------------------------------------------------------------------------
Let $u(x)\in C^7(R)$ and all the derivatives of $u(x)$ up to order $7$ belong to $L^1(R)$.  Then the following quasi-compact approximation has fourth order accuracy, i.e.,
\begin{equation}\label{2.1}
 P_x\,{_{-\infty}D_x^\alpha} u(x)
 =\mu_1 \Delta_1^\alpha u(x)+ \mu_0 \Delta_0^\alpha u(x)+ \mu_{-1} \Delta_{-1}^\alpha u(x) +O(h^4),
\end{equation}
where $P_x=1+h^2b_2^\alpha\delta_x^2$, called CWSGD operator; $\delta_x^2$ is the centered difference operator; and the coefficients $ b_2^\alpha  $, $\mu_1$, $\mu_0$, and $\mu_{-1}$ are the functions of $ \alpha$ and
\begin{equation}\label{coeff4}
\left\{
\begin{array}{lc}
\displaystyle \mu_1= (1 + \alpha)  (2 + \alpha)/12, \\\\
\displaystyle \mu_0 =- (-2 +\alpha) (2 + \alpha)/6,\\\\
\displaystyle  \mu_{-1} =(-2 + \alpha)  (-1 + \alpha)/12,\\\\
\displaystyle b_2^\alpha  =(4 + \alpha - \alpha^2)/24.
\end{array}
\right.
\end{equation}
%$\mu_r=\frac{(\alpha - 2 p) (\alpha - 2 q) (\alpha - p - q)}{2 (p - r) (q - r) (3 \alpha -
%   2 (p + q + r))}$, $\mu_p=-\frac{((-\alpha + 2 q) (\alpha - 2 r) (-\alpha + q + r)}{
% 2 (p - q) (p - r) (-3 \alpha + 2 (p + q + r))}$, $\mu_q=\frac{(2 p - \alpha) (p + r - \alpha) (-2 r + \alpha)}{2 (p - q) (q -
%   r) (2 p + 2 q + 2 r - 3 \alpha)}$, $b_2^\alpha  =  \frac{-24 p q r + 2 (p + q + 6 p q + r + 6 p r + 6 q r) \alpha -
% 3 (1 + 2 p + 2 q + 2 r)\alpha^2 + 3 \alpha^3}{2 (p + q + r) - 3 \alpha} $.
% \begin{equation}
%\begin{array}{lc}
%\displaystyle \mu_r=\frac{(\alpha - 2 p) (\alpha - 2 q) (\alpha - p - q)}{2 (p - r) (q - r) (3 \alpha -
%   2 (p + q + r))}, \\\\
%\displaystyle \mu_p=-\frac{((-\alpha + 2 q) (\alpha - 2 r) (-\alpha + q + r)}{
% 2 (p - q) (p - r) (-3 \alpha + 2 (p + q + r))},\\\\
%\displaystyle  \mu_q=\frac{(2 p - \alpha) (p + r - \alpha) (-2 r + \alpha)}{2 (p - q) (q -
%   r) (2 p + 2 q + 2 r - 3 \alpha)}\\\\
%\displaystyle b_2^\alpha  =  \frac{-24 p q r + 2 (p + q + 6 p q + r + 6 p r + 6 q r) \alpha -
% 3 (1 + 2 p + 2 q + 2 r)\alpha^2 + 3 \alpha^3}{2 (p + q + r) - 3 \alpha} .
%\end{array}
%\end{equation}
\end{thm}
In fact, under the assumptions of the theorem, we know that for any fixed order $\alpha$ and the coefficients $\mu_1 $, $\mu_0 $, and $\mu_{-1} $
    the following equalities hold.
\begin{equation}\label{schm11}
\begin{array}{l}
\displaystyle\mu_1 \Delta_1^\alpha u(x)+ \mu_0 \Delta_0^\alpha u(x)+ \mu_{-1} \Delta_{-1}^\alpha u(x)\\\\
 \displaystyle ={_{-\infty}D}_x^\alpha u(x)+b_2^\alpha   \,_{-\infty}D_x^{\alpha+2}u(x)h^2+O(h^4)\\\\
  \displaystyle  =\left(1+h^2b_2^\alpha \frac{\partial^2}{\partial x^2}\right)\,{_{-\infty}D_x}^\alpha u(x) +O(h^4)\\\\
   \displaystyle  =P_x\,{_{-\infty}D_x}^\alpha u(x)  +O(h^4),
\end{array}
\end{equation}
where $b_2^\alpha=\mu_1a_{1,2}^\alpha +\mu_0a_{0,2}^\alpha +\mu_{-1}a_{-1,2}^\alpha$.
 Then we get (\ref{2.1}). %the fourth order approximation
 Since $\delta_x^2 u(x)=(u(x-h)-2 u(x)+u(x+h))/h^2=\frac{\partial^2 u(x)}{\partial x^2}+O(h^2)$, we have  for any function $u$,
$$
P_x u=\left(1+h^2b_2^\alpha \frac{\partial^2}{\partial x^2}\right)u+O(h^4).
$$
 In a similar way, we can derive the fourth order quasi-compact approximation for the right Riemann-Liouville fractional derivative:
\begin{equation}
 P_x\, {_x D^\alpha_{+\infty}} u(x)
  =\mu_1 \Lambda_1^\alpha u(x)+ \mu_0 \Lambda_0^\alpha u(x)+ \mu_{-1} \Lambda_{-1}^\alpha u(x) +O(h^4).
\end{equation}
For $u(x)$ defined in a bounded interval, supposing its zero extension to $R$ satisfies the assumptions of Theorem \ref{th2.1}, the following approximations hold:
\begin{equation}\label{fors11}
P_x\,{_{a}D}^\alpha_x u(x)=\mu_1 \tilde{\Delta}_1^\alpha u(x)+ \mu_0 \tilde{\Delta}_0^\alpha u(x)+ \mu_{-1} \tilde{\Delta}_{-1}^\alpha u(x)+O(h^4)
\end{equation}
and
\begin{equation}\label{fors12}
P_x\,{_{x}D}^\alpha_b u(x)=\mu_1 \tilde{\Lambda}_1^\alpha u(x)+ \mu_0 \tilde{\Lambda}_0^\alpha u(x)+ \mu_{-1} \tilde{\Lambda}_{-1}^\alpha u(x)+O(h^4).
\end{equation}

Now using the CWSGD operator, we solve a two-point boundary value problem to numerically verify the above statements.

\begin{example}\label{exforth1}
Consider the steady state fractional diffusion problem
\[ _0D_x^\alpha u(x)=\frac{720 x^{6 -  \alpha}}{\Gamma(7 - \alpha)} ,  \quad x\in(0,1),\]
with  $1<\alpha<2$  and the boundary conditions $u(0)=0$, $u(1)=1$. Its exact solution is $u(x)=x^{6}$.
 \end{example}
\begin{table}[htbp]
\centering\small
\caption{Numerical errors and convergence rates in $L_\infty$ norm and $L_2$ norm by using (\ref{fors11}) to solve Example \ref{exforth1}, where $U$ denotes  the numerical solution and   $h_x$ is the space step size.}\vspace{1em}  {\begin{tabular}{@{}cccccc@{}} \hline
 $\alpha$ & $h_x$ & $\|u-U\|_2$ &rate & $\|u-U\|_\infty$ &rate  \\
 \hline
1.1&$1/8$ & $   6.0879e-04  $ &   $         $ & $ 1.0551e-03  $ &   $ $  \\
&$1/16 $&  $  2.7715e-05   $ &   $ 4.4572  $  &  $   5.1569e-05  $ &   $ 4.3548  $  \\
&$ 1/32$ &  $  1.5024e-06  $  &   $ 4.2054  $   & $  2.8244e-06  $  &    $ 4.1905   $  \\
&$ 1/64$ &  $ 9.0430e-08   $  &   $ 4.0543  $   &  $ 1.6385e-07   $  &    $   4.1075 $  \\
&$1/128$ &  $ 5.5808e-09   $  &   $   4.0183  $  &  $   9.5651e-09 $  &   $ 4.0984  $ \\
 \hline
1.5&$1/8$ & $    2.9459e-04  $ &   $          $ & $  3.9380e-04  $  &  $  $   \\
&$1/16 $&  $   1.8470e-05    $ &   $  3.9955  $    &  $ 2.4150e-05    $ &   $   4.0274   $  \\
&$ 1/32$ &  $  1.1590e-06    $  &  $  3.9942  $    &  $  1.5252e-06   $  &   $ 3.9850   $  \\
&$ 1/64$ &  $    7.2639e-08  $  &   $  3.9960  $   &  $ 9.5671e-08   $  &    $ 3.9948    $  \\
&$1/128$ &  $  4.5471e-09    $  &   $ 3.9977   $   &  $  5.9911e-09   $  &   $   3.9972  $ \\
\hline
1.9&$1/8$ & $ 1.1926e-04    $ &   $         $ & $    1.6198e-04   $ &   $      $  \\
&$1/16 $&  $ 7.4913e-06     $ &   $  3.9927 $  &  $   1.0174e-05   $ &   $   3.9928  $  \\
&$ 1/32$ &  $ 4.6919e-07    $  &  $  3.9970 $   &  $ 6.3722e-07    $  &    $ 3.9970   $  \\
&$ 1/64$ &  $  2.9352e-08   $  &  $  3.9986  $   &  $   3.9899e-08  $  &    $   3.9974  $  \\
&$1/128$ &  $  1.8352e-09   $  &  $ 3.9994   $  &  $    2.4947e-09   $  &   $  3.9994    $ \\
\hline
\end{tabular}\label{tab1}}
\end{table}

Using the quasi-compact scheme (\ref{fors11}) to solve Example \ref{exforth1} leads to the desired convergence orders; see Table \ref{tab1}.

%\begin{lem}

\subsection{ Fifth order quasi-compact approximation to the Riemann-Liouville fractional derivative}
%In this subsection, we give a fifth order quasi-compact approximation which   have a following form for solving some elliptic equation and ordinary differential equation.
In this subsection, we present a fifth order quasi-compact approximation given as follows.

%Similar to the fourth order approximations for the fractional derivatives, we give a fifth compact approximation.
\begin{thm}\label{th2.2}
Let $u(x)\in C^8(R)$ and all the derivatives of $u(x)$ up to order $8$ belong to $L^1(R)$.  Then the following quasi-compact approximation has fifth order accuracy, i.e.,
\begin{equation}\label{fif}
\displaystyle P_x^5\,  _{-\infty}D^\alpha_x u(x)
  =\mu_1 \Delta_1^\alpha f(x)+ \mu_0 \Delta_0^\alpha f(x)+ \mu_{-1} \Delta_{-1}^\alpha f(x) +O(h^5),
\end{equation}
where $P_x^5\,  _{-\infty}D^\alpha_x u(x) =\gamma_1\,  _{-\infty}D^\alpha_x u(x-h)+ \,  _{-\infty}D^\alpha_x u(x) +\gamma_2\,  _{-\infty}D^\alpha_x u(x+h)$, called 5-CWSGD operator, and
%the coefficients $ m$, $ n$, $\mu_1$, $\mu_0$ and $\mu_{-1}$ are
\begin{equation}\label{fifcoef}
\left\{
\begin{array}{lc}
\displaystyle \gamma_1=\frac{350 + 331 \alpha - 15 \alpha^2 - 75 \alpha^3 - 15 \alpha^4}{1724 - 2\alpha - 570\alpha^2 - 30 \alpha^3 + 30\alpha^4},\\\\
\displaystyle \gamma_2 = \frac{566 - 329\alpha - 135 \alpha^2 + 105\alpha^3 - 15 \alpha^4}{1724 - 2 \alpha - 570\alpha^2 - 30\alpha^3 + 30 \alpha^4},\\\\
\displaystyle \mu_1=\frac{566 + 329 \alpha - 135 \alpha^2 - 105 \alpha^3 - 15\alpha^4}{1724 - 2 \alpha - 570\alpha^2 - 30 \alpha^3 + 30\alpha^4},\\\\
\displaystyle  \mu_0=\frac{862 + \alpha- 285\alpha^2 + 15\alpha^3 + 15\alpha^4}{862 - a - 285 \alpha^2 - 15 \alpha^3 + 15\alpha^4},\\\\
\displaystyle \mu_{-1}=\frac{350 - 331\alpha - 15 \alpha^2 + 75 \alpha^3 - 15 \alpha^4}{1724 - 2\alpha - 570 \alpha^2 -  30 \alpha^3 + 30\alpha^4}.
\end{array}
\right.
\end{equation}
\end{thm}
The way of deriving (\ref{fif}) is similar to the derivation of the fourth order quasi-compact approximation. On one hand, from (\ref{scheme0}), we know for different parameter $p\in\{1,0,-1\}$ there exist three equalities
 \begin{equation} \label{2.2.1}
\Delta_p^\alpha u(x)={_{-\infty}D}^\alpha_x u(x)+\sum\limits_{k=1}^{4}a_{p, k}^\alpha \, _{-\infty}D_x^{\alpha+k}u(x)h^k+O(h^5),\quad p=1,0,-1.
\end{equation}
On the other hand, in view of the Taylor expansion we know
 \begin{equation}\label{2.2.2}
\begin{array}{llll}
\displaystyle  {_{-\infty}D}^\alpha _xu(x-h) ={_{-\infty}D}^\alpha_x u(x)+(-1)^k\sum\limits_{k=1}^{4}\frac{1}{k!}\,{_{-\infty}D}^{\alpha+k}_x u(x)h^k+O(h^5),\\\\
 \displaystyle  {_{-\infty}D}^\alpha _xu(x+h) ={_{-\infty}D}^\alpha_x u(x)+\sum\limits_{k=1}^{4}\frac{1}{k!}\,{_{-\infty}D}^{\alpha+k}_x u(x)h^k+O(h^5).
\end{array}
\end{equation}
So in order to get the fifth order quasi-compact approximation,  combining  (\ref{2.2.1}) and (\ref{2.2.2}), we need to eliminate the low order terms corresponding to $h^k$ ($k=1, 2, 3, 4$), which can be done by solving the algebraic equation
%After combination with (\ref{scheme0}), we get the fifth order approximation at any fixed  $\alpha$ there exist some constants $ m$, $ n$,  $\mu_r $, $\mu_p $ and $\mu_q $
% which    satisfies the following equalities
\begin{equation}\label{fifcoef5}
\left\{
\begin{array}{lc}
\displaystyle \mu_1+ \mu_0 +\mu_{-1}-\gamma_1-\gamma_2=1, \\\\
\displaystyle \mu_1a_{1,1}^\alpha +\mu_0a_{0,1}^\alpha +\mu_{-1}a_{-1,1}^\alpha+\gamma_1-\gamma_2 =0 ,\\\\
\displaystyle \mu_1a_{1,2}^\alpha +\mu_0a_{0,2}^\alpha +\mu_{-1}a_{-1,2}^\alpha-\gamma_1/2-\gamma_2/2 =0,\\\\
\displaystyle \mu_1a_{1,3}^\alpha +\mu_0a_{0,3}^\alpha +\mu_{-1}a_{-1,3}^\alpha +\gamma_1/3!-\gamma_2/3!=0 ,\\\\
\displaystyle \mu_1a_{1,4}^\alpha +\mu_0a_{0,4}^\alpha +\mu_{-1}a_{-1,4}^\alpha -\gamma_1/4!-\gamma_2/4!=0.
\end{array}
\right.
\end{equation}
Eq. (\ref{fifcoef}) is the solution of (\ref{fifcoef5}).
Then we get Theorem  \ref{th2.2}. Next we utilize the 5-CWSGD operator to solve Example  \ref{ex2}; and the numerical results are presented in Table \ref{tab2}, from which the accuracy of the 5-CWSGD operator is verified.

\begin{example}\label{ex2}
We again consider the steady state fractional diffusion problem simulated in Example \ref{exforth1}, i.e.,
\[ _0D_x^\alpha u(x)=\frac{720 x^{6 -  \alpha}}{\Gamma(7 - \alpha)} ,  \quad x\in(0,1),\]
with  $1<\alpha<2$  and the boundary conditions $u(0)=0$, $u(1)=1$; and the exact solution $u(x)=x^{6}$.
\end{example}
\begin{table}[htbp]
\centering\small
\caption{Numerical errors and convergence rates in $L_\infty$ norm and $L_2$ norm of scheme (\ref{fif}) to solve Example \ref{ex2}, where $U$ denotes  the numerical solution and   $h_x$ is space step size.}\vspace{1em}  {\begin{tabular}{@{}cccccc@{}} \hline
 $\alpha$ & $h_x$ & $\|u-U\|_2$ &rate & $\|u-U\|_\infty$ &rate  \\
\hline
1.1&$1/8$ & $   2.3456e-05   $ &   $         $ & $  5.2058e-05  $ &   $ $  \\
&$1/16 $&  $    6.8783e-07  $ &   $   5.0918  $  &  $   1.6758e-06  $ &   $  4.9572   $  \\
&$ 1/32$ &  $  2.0903e-08   $  &   $  5.0403   $   & $  5.3410e-08   $  &    $ 4.9716     $  \\
&$1/ 64$ &  $  6.4355e-10   $  &   $ 5.0215   $   &  $  1.6852e-09   $  &    $ 4.9861      $  \\
&$1/128$ &  $   1.9956e-11  $  &   $     5.0112  $  &  $   5.2916e-11  $  &   $  4.9931   $ \\
 \hline
1.5&$1/8$ & $    9.0595e-06   $ &   $          $ & $  1.9904e-05   $  &  $  $   \\
&$1/16 $&  $ 2.8200e-07  $ &   $  5.0057   $    &  $    6.7018e-07  $ &   $   4.8924     $  \\
&$ 1/32$ &  $  8.9299e-09   $  &  $  4.9809    $    &  $   2.2033e-08   $  &   $  4.9268    $  \\
&$ 1/64$ &  $   2.8313e-10    $  &   $ 4.9791   $   &  $  7.1095e-10   $  &    $ 4.9538      $  \\
&$1/128$ &  $   8.9603e-12    $  &   $   4.9818   $   &  $ 2.2661e-11     $  &   $  4.9714     $ \\
\hline
%1.9&$1/8$ & $  2.3095e-05   $ &   $         $ & $  3.2028e-05     $ &   $      $  \\
%&$1/16 $&  $  1.1473e-06    $ &   $  4.3312  $  &  $   1.6323e-06    $ &   $   4.2944   $  \\
%&$1/ 32$ &  $    5.9763e-08 $  &  $  4.2629   $   &  $   8.8575e-08   $  &    $  4.2039    $  \\
%&$ 1/64$ &  $  2.7773e-09    $  &  $  4.4275    $   &  $  4.3459e-09    $  &    $  4.3492     $  \\
%&$1/128$ &  $    1.1193e-10  $  &  $    4.6330 $  &  $   1.8425e-10    $  &   $   4.5599     $ \\
%\hline
\end{tabular}}\label{tab2}
\end{table}

\begin{rem}
As the fifth order quasi-compact scheme is not stable in solving the time-dependent space fractional differential equation, we detailedly discuss the fourth order quasi-compact schemes in Sections 3 and 4.
\end{rem}

%---------------------------------------------------------------------------------------------------------------------------------------------------------------------

\section{Quasi-compact scheme for one dimensional space fractional diffusion equation}\label{sec1}
Based on the fourth order quasi-compact discretization to the Riemann-Liouville space fractional derivative,
we develop the Crank-Nicolson quasi-compact scheme of the two-sided space fractional diffusion equations.
%\subsection{}
Here we consider the initial boundary value problem of the space fractional diffusion equation
\begin{equation}\label{equ1}
\left\{
\begin{array}{lll}
\displaystyle \frac{\partial u(x,t)}{\partial t} =K_1\, _aD_x^\alpha u(x,t)+K_2\,_xD_b^\alpha u(x,t)+f(x,t),  &(x,t)\in (a,b)\times(0,T], \\\\
\displaystyle  u(x,0)=u_0(x),&x\in[a,b],\\\\
\displaystyle   u(a,t)=\phi_a(t),\,\,u(b,t)=\phi_b(t),& t\in[0,T],
\end{array}
\right.
\end{equation}
where $1<\alpha\leq2$. The diffusion coefficients $K_1$ and $K_2$ are nonnegative constants and they satisfy $K_1^2+ K_2^2\neq0$.
 If $K_1\neq0$, then $\phi_a(t)\equiv0$ and $K_2\neq 0$, then $\phi_b(t)\equiv0$.  %Next we discretize the equations (\ref{equ1}) by the proposed scheme.
 In the following analysis of the numerical method, we suppose that (\ref{equ1}) has an unique and sufficiently smooth solution.

\subsection{CN-CWSGD scheme}
  The time interval $[0,T]$ is partitioned  into a uniform mesh with the step size $\tau =T/N$
 and  the space interval $[a,b]$ into another uniform mesh with the step sized $h=(b-a)/M$, where $N, M$ are two positive integers.
 Then the set of grid points can be  denoted by $x_j=a+jh$ $(0\leq j\leq M)$ and $t_n=n\tau$ $(0\leq n\leq N)$.
 Let $u_j^n=u(x_j,t_n)$,  $t_{n+1/2}=(t_n+t_{n+1})/2$, and $f^{n+1/2}_j=f(x_j,t_{n+1/2})$ for $0\leq n\leq N-1$.
 The maximum norm and the discrete $L_2$ norm are defined as
\begin{equation}
\|u\|_\infty=\max\limits_{1\leq j\leq M-1}|u_j|,\quad \|u\|^2= h\sum_{j=1}^{M-1}u_j^2.
\end{equation}
We use the Crank-Nicolson technique for the time discretization of (\ref{equ1}) and get
\begin{equation}
\begin{array}{lll}
\displaystyle
\frac{u_j^{n+1}-u_j^n}{\tau}=\frac{1}{2}\left(K_1(_aD_x^\alpha u)_j^n + K_1(_aD_x^\alpha u)_j^{n+1}  +K_2(_xD_b^\alpha u)_j^n + K_2(_xD_b^\alpha u)_j^{n+1} \right) \\\\\displaystyle~~~~~~~~~~~~~~~~
+f^{n+1/2}_j+O(\tau^2).
\end{array}
\end{equation}
In space, the fourth order quasi-compact discretizations are used to approximate the Riemann-Liouville fractional derivatives. This implies that
\begin{equation}
\begin{array}{l }
\displaystyle
P_x\frac{u_j^{n+1}-u_j^n}{\tau}
=\frac{K_1\tau}{2 } \,_{L}D_{h}^\alpha u_j^{n} +\frac{K_2\tau}{2 } \,_{R}D_{h}^\alpha u_j^{n}+
\frac{K_1\tau}{2 } \,_{L}D_{h}^\alpha u_j^{n+1} +\frac{K_2\tau}{2 } \,_{R}D_{h}^\alpha u_j^{n+1}\\\\\displaystyle~~~~~~~~~~~~~~~~~~~~
+ P_x f^{n+1/2}_j  + R^{n+1/2}_j,
\end{array}
\end{equation}
where
$$\,_{L}D_{h}^\alpha   u_{j }^{n} =: \mu_1 \tilde{\Delta}_1^\alpha u_{j }^{n}+ \mu_0 \tilde{\Delta}_0^\alpha u_j^{n}+ \mu_{-1} \tilde{\Delta}_{-1}^\alpha u_j^{n}=\frac{1}{h^\alpha}\sum_{k=0}^{j+1}w_{k}^{(\alpha)}u_{j-k+1}^{n},$$
 $$\,_{R}D_{h}^\alpha u_j^{n}=:\mu_1 \tilde{\Lambda}_1^\alpha u_j^{n}+ \mu_0 \tilde{\Lambda}_0^\alpha u_j^{n}+ \mu_{-1} \tilde{\Lambda}_{-1}^\alpha u_j^{n}=\frac{1}{h^\alpha}\sum_{k=0}^{M-j+1}w_{k}^{(\alpha)}u_{j+k-1}^{n},$$
%When  third  order scheme () is used  to approximation the space derivative,    the coefficients $w_{3,k}^{(\alpha)}$ satisfy
%\begin{equation}\label{wl3}
%\sum_{k=0}^{i+1}w_{3,k}^{(\alpha)}u_{i-k+1}^{n}=\mu_p\sum_{k=0}^{i+p}g_k^{(\alpha)}u_{i-k+p}^{n}+\mu_q\sum_{k=0}^{i+q}g_k^{(\alpha)}u_{i-k+q}^{n},\quad (p,q)=(1,0),(1,-1)
%\end{equation}
%When  fourth  order scheme (\ref{fors}) is used  to approximation the space derivative,
the coefficients $w_0^{(\alpha)}=\mu_1g_0^{(\alpha)}$,  $w_1^{(\alpha)}=\mu_0g_0^{(\alpha)}+\mu_1g_1^{(\alpha)}$, and $w_k^{(\alpha)}=\mu_1g_k^{(\alpha)}+\mu_0g_{k-1}^{(\alpha)}+\mu_{-1}g_{k-2}^{(\alpha)}$, $k=2,\cdots, M$
  and   $R^{n+1/2}_j\leq C(\tau^2+h^4) $.
Then the above equation can be rewritten as
\begin{equation}\label{fore}
\begin{array}{lll}
\displaystyle
 P_xu_j^{n+1}-\frac{K_1\tau}{2 } \,_{L}D_{h}^\alpha u_j^{n+1} -\frac{K_2\tau}{2 } \,_{R}D_{h}^\alpha u_j^{n+1}\\\\\displaystyle
=P_xu_j^{n}+\frac{K_1\tau}{2 } \,_{L}D_{h}^\alpha u_j^{n} +\frac{K_2\tau}{2 } \,_{R}D_{h}^\alpha u_j^{n}
+\tau P_x f_j^{n+1/2}+\tau R^{n+1/2}_j.
\end{array}
\end{equation}
Denoting $U_j^n$ as the numerical approximation of $u_j^n$, we obtain the Crank-Nicolson quasi-compact scheme for (\ref{equ1})
\begin{equation}\label{fourthccd}
\begin{array}{lll}
\displaystyle
 P_xU_j^{n+1}-\frac{K_1\tau}{2 } \,_{L}D_{h}^\alpha U_j^{n+1} -\frac{K_2\tau}{2 } \,_{R}D_{h}^\alpha U_j^{n+1} \\\\\displaystyle
=P_xU_j^{n}+\frac{K_1\tau}{2 } \,_{L}D_{h}^\alpha U_j^{n} +\frac{K_2\tau}{2 } \,_{R}D_{h}^\alpha U_j^{n}
+\tau P_x f_j^{n+1/2}.
\end{array}
\end{equation}
For convenience, the approximation scheme (\ref{fourthccd}) can be written in matrix form
\begin{equation}\label{matrixfourth}
 (P_\alpha-B_\alpha)U^{n+1}=(P_\alpha+B_\alpha)U^{n}+\tau F^n+H^n,
 \end{equation}
 where  $B_\alpha=\frac{\tau}{2h^\alpha}(K_1A_\alpha+K_2A_\alpha^T)$,
 $
U^n=(U^n_1,U^n_2,\cdots ,U^n_{M-1})^T,\quad F^n=(f^{n+1/2}_1,f^{n+1/2}_2,\cdots ,f^{n+1/2}_{M-1})^T
$,
  $A_\alpha$ is given by
\begin{equation}\label{mA}     %开始数学环境
\displaystyle A_\alpha=\left(                 %左括号
  \begin{array}{ccccc}   %该矩阵一共3列，每一列都居中放置
    w_1^{(\alpha)}    & w_0^{(\alpha)}  &                 &   &    \\  %第一行元素
   w_2^{(\alpha)}    & w_1^{(\alpha)} &  w_0^{(\alpha)}   &    &     \\  %第二行元素
       \vdots      &     w_2^{(\alpha)} &  w_1^{(\alpha)}    &    & \\  %第二行元素
    w_{M-2}^{(\alpha)} &  \cdots      &\ddots   & \ddots   & w_0^{(\alpha)}   \\  %第一行元素
    w_{M-1}^{(\alpha)}&w_{M-2}^{(\alpha)} & \cdots  & w_2^{(\alpha)}  & w_1^{(\alpha)} \\  %第二行元素
  \end{array}
\right),                 %右括号
\end{equation}
\[       %开始数学环境
\displaystyle P_\alpha=\left(                 %左括号
  \begin{array}{ccccc}   %该矩阵一共3列，每一列都居中放置
    1-2b_2^\alpha     & b_2^\alpha   &     &   &    \\  %第一行元素
   b_2^\alpha   & (1-2b_2^\alpha  ) & b_2^\alpha   &    &     \\  %第二行元素
          &      &\cdots &    &     \\  %第二行元素
          &      &b_2^\alpha   & 1-2b_2^\alpha   & b_2^\alpha   \\  %第一行元素
          &    &   & b_2^\alpha   & 1-2b_2^\alpha   \\  %第二行元素
  \end{array}
\right)                 %右括号
\]
and
\begin{equation}      %开始数学环境
\begin{array}{lll}
H^n&=&\left(                 %左括号
 \begin{array}{c} %该矩阵一共3列，每一列都居中放置
     b_2^\alpha      \\  %第一行元素
       0   \\  %第二行元素
       \vdots        \\  %第二行元素
        0      \\  %第二行元素
  \end{array}
\right)
(U^n_0-U^{n+1}_0)+\frac{\tau}{2h^\alpha}\left(                 %左括号
  \begin{array}{c}   %该矩阵一共3列，每一列都居中放置
   K_1w_2^{(\alpha)}+  K_2w_0^{(\alpha) }  \\  %第一行元素
        K_1w_3^{(\alpha)}   \\  %第二行元素
       \vdots        \\  %第二行元素
        K_1w_{M-1}^{(\alpha) }    \\  %第一行元素
         K_1w_{M}^{(\alpha) }   \\  %第二行元素
  \end{array}
\right)(U^n_0+U^{n+1}_0)\\\\
 && \displaystyle +              %右括号
\left(                 %左括号
  \begin{array}{c}   %该矩阵一共3列，每一列都居中放置
     0    \\  %第一行元素
       \vdots   \\  %第二行元素
       0        \\  %第二行元素
        b_2^\alpha       \\  %第二行元素
  \end{array}
\right)
(U^n_M-U^{n+1}_M)+\frac{\tau}{2h^\alpha}\left(                 %左括号
  \begin{array}{c}   %该矩阵一共3列，每一列都居中放置
    K_2w_{M}^{(\alpha)}  \\  %第一行元素
     K_2w_{M-1}^{(\alpha) } \\  %第二行元素
       \vdots        \\  %第二行元素
     K_2w_{3}^{(\alpha) } \\  %第一行元素
      K_1w_ 0^{(\alpha)} +K_2w_{2}^{(\alpha)}     \\  %第二行元素
  \end{array}
\right)(U^n_M+U^{n+1}_M).
%\\\\
%\displaystyle~~~~+\frac{\tau}{2h^\alpha}\left(                 %左括号
%  \begin{array}{c}   %该矩阵一共3列，每一列都居中放置
%    0 \\  %第一行元素
%     0 \\  %第二行元素
%       \vdots        \\  %第二行元素
%     0 \\  %第一行元素
%      K_1w_ 0^{(\alpha)}     \\  %第二行元素
%  \end{array}
%\right)(U^n_M-U^{n+1}_M)
 \end{array}
\end{equation}

%\begin{thm}\label{Th3.1}%---------------------------------------------------------------------thm----
%
%\end{thm}%-------------------------------------------------------thm-----------
%
%\begin{proof}
%\Gamma
%\end{proof}

%-------------------------------------------------------------------------------------------------------------------------------------------------------------------
\subsection{Stability and convergence analysis}

In this subsection, we prove that the CN quasi-compact scheme has fourth order accuracy in space and is unconditionally stable. Now we give some important lemmas to be used in the analyses.
%we verify the C-N quasi-compact scheme is unconditionally stable  and  convergent, which we discuss in $L_2$ norm. First we give some corresponding lem.
%\begin{lem}(\cite{bhatia2009positive})\label{sy}
%Let $A\in R^{n\times n}$.   And it satisfies $v^T Av >0$ for all real nonzero vectors $v$, if and only if its symmetric part $H=\frac{A+A^T}{2}$ is positive definite.
%\end{lem}

\begin{lem} (\cite{chan2007introduction})\label{sy1}
 Let $H$ be a Toeplitz matrix with a generating function $f\in C_{2\pi}$. Let $ \lambda_{ \min}(H)$ and $ \lambda_{ \max}(H)$  denote the smallest and
largest eigenvalues of  $H$, respectively. Then we have
\[ f_{ \min}\leq\lambda_{ \min}(H)\leq\lambda_{ \max}(H)\leq f_{ \max},\]
where $f_{ \min}$ and $f_{ \max} $ denote the minimum and maximum values of $f(x)$, respectively.
In particular, if $f_{ \max}\leq0$ and $f_{\min} \neq f_{ \max}$, then $H $ is negative definite.
\end{lem}

\begin{lem}(\cite{bhatia2009positive})\label{le3.3}
Let $A$ be a positive semi-definite matrix. Then there exists a unique $n$-square positive semi-definite matrix $B$ such that $B^2=A$. Such a matrix $B$ is called the square root of $A$, denoted by $A^\frac{1}{2}$.
\end{lem}
%\begin{lem}(\cite{marcus})\label{LS}
%The matrix $D\in R^{n\times n}$ is asymptotically stable, if and only if to the Lyapunov equation
%\[DX+XD^T=C,\]
%there exists a symmetric  and positive (or negative) definite solution, where $C=C^T\in R^{n\times n}$ is a negative (or negative) matrix. And if all
%eigenvalues of  a matrix $D$ have real parts in the open left half-plane, i.e.
%$Re(\mu(D))<0$, then the matrix $D$ is called asymptotically stable.
%$X\in R^{n\times n}$,
%\end{lem}
\begin{thm}\label{ka}
The matrix  $A_\alpha+A_\alpha^T$  is negative definite, and $B_\alpha+B_\alpha^T$ is also  negative definite, where  $A_\alpha$ is given by (\ref{mA}) and  $B_\alpha$ defined in (\ref{matrixfourth}).
%Let the matrix $A_\alpha$ be given by (\ref{mA}). %, where the diagonals $\{w_k^{(\alpha,\lambda)}\}_{k=0}^{M-1}$ are the coefficients given in (\ref{wl4}).
%Then we know that the matrix  $A_\alpha+A_\alpha^T$  is negative definite, when $\alpha\in(1,2]$. Moreover, matrix $K_1A_\alpha+K_2A_\alpha^T+( K_1A_\alpha+K_2A_\alpha^T)^T$ is also  negative definite, where $K_1$ and $ K_2 $ are defined in (\ref{equ1}).
\end{thm}
%In fact,  we know that the generating functions of $\frac{A+A^T}{2}$ is $f(\alpha;x)=\frac{f_A(x)+f_{A^T}(x)}{2}$\cite{chan2007introduction,Tian2012},
%and $f(\alpha;x)= $

In fact, the generating function \cite{chan2007introduction} of $A+A^T$ satisfies
\begin{equation}\label{GeneratFunc}
\begin{array}{lll}
\displaystyle~~~~
f(\alpha,x)\\
\displaystyle= f_{A_\alpha}(x)+f_{A_\alpha^T}(x)=\left(\sum_{k=0}^{\infty}w_{k}^{(\alpha)} e^{-i(k-1)x }+\sum_{k=0}^{\infty}w_{k}^{(\alpha)} e^{i(k-1)x}\right)\\\\
\displaystyle = \mu_1 \left( \sum_{k=0}^{\infty}g_k^{(\alpha)}e^{-i(k-1)\sigma }+ \sum_{k=0}^{\infty}g_k^{(\alpha)}e^{ i(k-1)\sigma }\right)+ \mu_0 \left( \sum_{k=0}^{\infty}g_k^{(\alpha)}e^{-ik\sigma }+ \sum_{k=0}^{\infty}g_k^{(\alpha)}e^{ ik\sigma }\right)\\\\
\displaystyle  ~~~~+ \mu_{-1} \left( \sum_{k=0}^{\infty}g_k^{(\alpha)}e^{-i(k+1)\sigma }+ \sum_{k=0}^{\infty}g_k^{(\alpha)}e^{ i(k+1)\sigma }\right)\\\\
\displaystyle   = \mu_1  ( (1-e^{-i\sigma })^\alpha e^{i\sigma }+ (1-e^{ i\sigma })^\alpha e^{-i\sigma })+ \mu_0 (  (1-e^{-i\sigma })^\alpha + (1-e^{ i\sigma })^\alpha)\\\\
\displaystyle  ~~~~+ \mu_{-1}  ( (1-e^{-i\sigma })^\alpha e^{-i\sigma } + (1-e^{ i\sigma })^\alpha e^{ i\sigma } )\\\\
 \displaystyle
= (2\sin(\frac{\sigma }{2}))^\alpha(\mu_1(  e^{i(\frac{\alpha\pi}{2}-\frac{\alpha\sigma }{2}+\sigma  )}+  e^{-i(\frac{\alpha\pi}{2}-\frac{\alpha\sigma }{2}+\sigma  )}) +  \mu_0(  e^{i(\frac{\alpha\pi}{2}-\frac{\alpha\sigma }{2})} +  e^{-i(\frac{\alpha\pi}{2}-\frac{\alpha\sigma }{2})} )\\\\
\displaystyle  ~~~~+  \mu_{-1}(   e^{i(\frac{\alpha\pi}{2}-\frac{\alpha\sigma }{2}-\sigma   )}+   e^{-i(\frac{\alpha\pi}{2}-\frac{\alpha\sigma }{2}-\sigma   )}))\\\\
 \displaystyle   =2\left(2\sin(\frac{x }{2})\right)^\alpha \left( \mu_1 \cos(\frac{\alpha\pi}{2}-\frac{\alpha x}{2}+ x )+
   \mu_0 \cos(\frac{\alpha\pi}{2}-\frac{\alpha x }{2})+\mu_{-1} \cos(\frac{\alpha\pi}{2}-\frac{\alpha x }{2}-x )\right),
\end{array}
\end{equation}
where $f_{A_\alpha}(x)$ and $f_{A_\alpha^T}(x)$ denote the generating functions of the matrix $A_\alpha$ and $A_\alpha^T$, respectively.
Since $f(\alpha;x)$ is a real-valued and even function,  it's reasonable to consider its principal value on $[0,\pi]$.
\begin{figure}
  \centering
  \includegraphics[width=0.6\textwidth ]{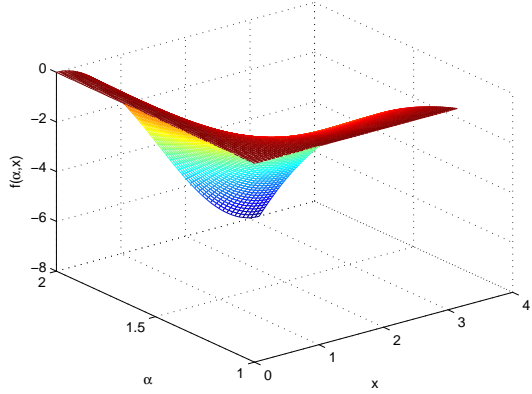}\\
  \caption{$f(\alpha;x)$ defined by (\ref{GeneratFunc}) for  $1 \le \alpha \le 2$ on $x\in[0,\pi]$. }\label{fig1}
\end{figure}
Together with  Fig. \ref{fig1},  we have that $f(\alpha;x)\leq0$ for $1\leq\alpha\leq2$ on   $[-\pi,\pi]$.
Then from Lemma \ref{sy1}, we know the matrix  $A_\alpha+A_\alpha^T$  is negative definite.
Rewriting  $B_\alpha+B_\alpha^T$ as $\frac{\tau}{2h^\alpha}(K_1(A_\alpha+A_\alpha^T)+K_2(A_\alpha^T+A_\alpha))$, %and  the Grenander-Szeg\"{o} theorem,
it can be clearly seen that $B_\alpha+B_\alpha^T$ is negative definite.

\begin{thm} \label{stability1d}%-------------------------------------------------------------------------------------------------
The difference scheme (\ref{fourthccd}) with $ \alpha\in(1,2)$ is unconditionally stable.
\end{thm}
\begin{proof}
Define the round-off error  as $\epsilon_j^n=U_j^n-\tilde{U}_j^n$,
where $\tilde{U}_j^n$ is the exact solution of the discretized equation (\ref{fourthccd})   and $ U_j^n $ is the numerical solution of the discretized equation (\ref{fourthccd}) obtained in finite precision arithmetic. Since $\tilde{U}_j^n$ satisfies the discretized equation exactly,  round-off error $\epsilon_j^n$ must also satisfy the discretized equation  {\cite{anderson1995computational}}.  Thus we obtain the  following  error equation
\begin{equation}\label{errorequation}
%\begin{array}{l}
%\displaystyle
 P_x\epsilon_j^{n+1}-\frac{K_1\tau}{2 } \,_{L}D_{h}^\alpha \epsilon_j^{n+1} -\frac{K_2\tau}{2 } \,_{R}D_{h}^\alpha \epsilon_j^{n+1}
=P_x\epsilon_j^{n}+\frac{K_1\tau}{2 } \,_{L}D_{h}^\alpha \epsilon_j^{n} +\frac{K_2\tau}{2 } \,_{R}D_{h}^\alpha \epsilon_j^{n}.
%\end{array}
\end{equation}
Since the  boundary conditions of error equation (\ref{errorequation}) are  $\epsilon_0^n=\epsilon_M^n=\epsilon_0^{n+1}=\epsilon_M^{n+1}=0$,
  we zero extend the solution of the problem (\ref{errorequation}) to the whole real line $R$. So it's reasonable to replace the symbols $j+1$ and $M-j+1$ in error equation (\ref{errorequation}) with $\infty$.
Now we have
\begin{equation}\label{extendederrorequation}
\begin{array}{l}
\displaystyle
b_2^\alpha  \epsilon_{j-1}^{n+1}+(1-2b_2^\alpha  )\epsilon_j^{n+1} +b_2^\alpha  \epsilon_{j+1}^{n+1}-\frac{K_1\tau}{2h^\alpha}\sum_{k=0}^{\infty}w_{k}^{(\alpha )}\epsilon_{j-k+1}^{n+1} -\frac{K_2\tau}{2h^\alpha}\sum_{k=0}^{\infty}w_{k}^{(\alpha )}\epsilon_{j+k-1}^{n+1} \\\\
\displaystyle
=b_2^\alpha  \epsilon_{j-1}^{n}+(1-2b_2^\alpha  )\epsilon_j^{n} +b_2^\alpha  \epsilon_{j+1}^{n}+\frac{K_1\tau}{2h^\alpha}\sum_{k=0}^{\infty}w_{k}^{(\alpha )}\epsilon_{j-k+1}^{n} +\frac{K_2\tau}{2h^\alpha}\sum_{k=0}^{\infty}w_{k}^{(\alpha )}\epsilon_{j+k-1}^{n}.
\end{array}
\end{equation}
Let $\epsilon_j^n=v^ne^{ij\sigma}$ be the solution of (\ref{extendederrorequation}), where $i=\sqrt{-1}$, $v^n$ is the amplitude at time level $n$ and
$\sigma(=2\pi h/k)$ is the phase angle with wavelength $k$. We just need to prove that the amplification factor $v(\sigma,\alpha)$ satisfies
the relation $|v(\sigma,\alpha) |\leq1$ for all $\sigma$ in $[-\pi,\pi]$. In fact, by substituting the expressions of $\epsilon_j^n(=v^ne^{ij\sigma})$ and $\epsilon_j^{n+1}(=v^{n+1}e^{ij\sigma})$ into (\ref{extendederrorequation}), we obtain the amplification factor of the  CN quasi-compact scheme
%It's easy to know that the initial condition can satisfy $\epsilon^0\in L^2(-\infty,+\infty)$.
%Denote the growth factor  $G(\sigma,\tau)$ satisfy $v^{n+1}(\sigma)=G(\sigma,\tau)v^{n}(\sigma)$, Then we have
$$
\begin{array}{lll}
\displaystyle v(\sigma,\alpha) =\frac{ 1-4b_2^\alpha  \sin^2\frac{\sigma }{2}  +\frac{K_1\tau}{2h^\alpha}\sum\limits_{k=0}^{\infty}w_{k}^{(\alpha)}e^{-i(k-1)\sigma } +\frac{K_2\tau}{2h^\alpha}\sum\limits_{k=0}^{\infty}w_{k}^{(\alpha)}e^{i(k-1)\sigma }  }
{ 1-4b_2^\alpha  \sin^2\frac{\sigma }{2}-\frac{K_1\tau}{2h^\alpha}\sum\limits_{k=0}^{\infty}w_{k}^{(\alpha)}e^{-i(k-1)\sigma }-\frac{K_2\tau}{2h^\alpha}\sum\limits_{k=0}^{\infty}w_{k}^{(\alpha)}e^{i(k-1)\sigma }  }\\\\
\displaystyle~~~~~~~~ ~~= \frac{Q_1(\sigma,\alpha)  +Q_2(\sigma,\alpha)  }
{ Q_1(\sigma,\alpha) - Q_2(\sigma,\alpha) },
\end{array}
$$
where  $Q_1(\sigma,\alpha)= 1-4b_2^\alpha  \sin^2\frac{\sigma }{2}$ and  $Q_2(\sigma,\alpha)=\frac{K_1\tau}{2h^\alpha}\sum\limits_{k=0}^{\infty}w_{k}^{(\alpha)}e^{-i(k-1)\sigma } +\frac{K_2\tau}{2h^\alpha}\sum\limits_{k=0}^{\infty}w_{k}^{(\alpha)}e^{i(k-1)\sigma }$.
% In view of the corresponding  coefficients $w_{k}^{(\alpha,\lambda)}$ given by (\ref{wl4}) and (\ref{coeff4}), we know
%When  third  order scheme () is used  to approximation the space derivative,    the coefficients $w_{3,k}^{(\alpha)}$ satisfy,\quad (p,q)=(1,0),(1,-1)
A straightforward calculation yields
\begin{equation}\label{q2}
\begin{array}{l}
\displaystyle~~~~ Q_2(\sigma,\alpha)\\\\
\displaystyle=\frac{K_1\tau}{2h^\alpha} \sum_{k=0}^{\infty}w_{k}^{(\alpha)} e^{-i(k-1)\sigma } +\frac{K_2\tau}{2h^\alpha}\sum_{k=0}^{\infty}w_{k}^{(\alpha)} e^{i(k-1)\sigma }\\\\
\displaystyle  =\frac{\mu_1\tau}{2h^\alpha}(K_1\sum_{k=0}^{\infty}g_k^{(\alpha)}e^{-i(k-1)\sigma }+K_2\sum_{k=0}^{\infty}g_k^{(\alpha)}e^{ i(k-1)\sigma })+\frac{\mu_0 \tau}{2h^\alpha}(K_1\sum_{k=0}^{\infty}g_k^{(\alpha)}e^{-i(k)\sigma }\\\\
\displaystyle~~~~ +K_2\sum_{k=0}^{\infty}g_k^{(\alpha)}e^{ i(k)\sigma })+\frac{\mu_{-1} \tau}{2h^\alpha}(K_1\sum_{k=0}^{\infty}g_k^{(\alpha)}e^{-i(k+1)\sigma }+K_2\sum_{k=0}^{\infty}g_k^{(\alpha)}e^{ i(k+1)\sigma })\\\\
\displaystyle  =\frac{\mu_1\tau}{2h^\alpha} (K_1(1-e^{-i\sigma })^\alpha e^{i\sigma }+K_2(1-e^{ i\sigma })^\alpha e^{-i\sigma })+\frac{\mu_0\tau}{2h^\alpha}( K_1(1-e^{-i\sigma })^\alpha +K_2(1-e^{ i\sigma })^\alpha)\\\\
\displaystyle~~~~ +\frac{\mu_{-1}\tau}{2h^\alpha} (K_1(1-e^{-i\sigma })^\alpha e^{-i\sigma } +K_2(1-e^{ i\sigma })^\alpha e^{ i\sigma } )\\\\
 \displaystyle
=\frac{\tau}{2h^\alpha}(2\sin(\frac{\sigma }{2}))^\alpha(\mu_1(K_1 e^{i(\frac{\alpha\pi}{2}-\frac{\alpha\sigma }{2}+\sigma  )}+K_2 e^{-i(\frac{\alpha\pi}{2}-\frac{\alpha\sigma }{2}+\sigma  )}) +  \mu_0(K_1 e^{i(\frac{\alpha\pi}{2}-\frac{\alpha\sigma }{2})}\\\\
\displaystyle~~~~+K_2 e^{-i(\frac{\alpha\pi}{2}-\frac{\alpha\sigma }{2})} ) +  \mu_{-1}( K_1 e^{i(\frac{\alpha\pi}{2}-\frac{\alpha\sigma }{2}-\sigma   )}+ K_2 e^{-i(\frac{\alpha\pi}{2}-\frac{\alpha\sigma }{2}-\sigma   )})).
\end{array}
\end{equation}
As $Q_1(\sigma,\alpha)$ is real-valued,
\[
|  v(\sigma,\alpha) |=\frac{|Q_1  +Q_2   |}
{ |Q_1  - Q_2  |  }= \sqrt{ \frac{ \left(Q_1   +Re(Q_2 )\right)^2+  (Im(Q_2 ) ) ^2  }
{ \left(Q_1  -Re(Q_2 )\right)^2+ \left(Im(Q_2 ) \right)^2 }},
\]
where $Re(Q_2 )$ and $Im(Q_2 )$ are real part and  imaginary part of $Q_2 $, respectively.
In order to   prove  that $| v(\sigma,\alpha)|\leq1$, we need to check
$$Q_1 \cdot Re(Q_2 )\leq 0.$$
Note that  $b_2^\alpha  =(4+\alpha-\alpha^2)/24\leq1/6$  for any $\alpha\in[1,2]$.
So $ Q_1=1-4b_2^\alpha  \sin^2(\frac{\sigma }{2})>0$.
Form (\ref{q2}), we know
\begin{equation*}
\begin{array}{lll}
  Re(Q_2 ) &=& \displaystyle \frac{(K_1+K_2)\tau }{2h^\alpha} (2\sin(\frac{\sigma } {2}))^\alpha  ( \mu_1 \cos(\frac{\alpha\pi}{2}-\frac{\alpha\sigma }{2}+\sigma )+\mu_0 \cos(\frac{\alpha\pi}{2}-\frac{\alpha\sigma }{2}) \\ \\ && \displaystyle
+\mu_{-1} \cos(\frac{\alpha\pi}{2}-\frac{\alpha\sigma }{2}-\sigma )) \\ \\
&=&  \displaystyle \frac{(K_1+K_2)\tau }{4h^\alpha} f(\alpha; \sigma),
\end{array}
\end{equation*}
where $f(\alpha; \sigma)$ is defined by (\ref{GeneratFunc}).
%In view of Figure \ref{fig1}, we know
%$f(\alpha,x)=(2\sin(\frac{\sigma }{2}))^\alpha   ( \mu_1 \cos(\frac{\alpha\pi}{2}-\frac{\alpha\sigma }{2}+\sigma )+  \mu_0 \cos(\frac{\alpha\pi}{2}-\frac{\alpha\sigma }{2})+\mu_{-1} \cos(\frac{\alpha\pi}{2}-\frac{\alpha\sigma }{2}-\sigma   ) )\leq0$.
Together  with $K_1+K_2> 0$ and Fig. \ref{fig1}, we obtain $ Re(Q_2 )\leq0 $.  Thus $Q_1 \cdot Re(Q_2 )\leq0 $. Then $| v(\sigma,\alpha)|\leq1$. So the C-N quasi-compact difference scheme is unconditionally stable.
\end{proof}
\begin{thm} \label{convergence}%-------------------------------------------------------------------------------------------------
Let $u(x_j,t_n)$ be the exact solution of (\ref{equ1}), and $U_j^n$ the solution of the given finite difference scheme (\ref{fourthccd}). Then we have
\[\|u(x_j,t_n)-U_j^n\|\leq C(\tau^2+h^4),\]
for all $1\leq n\leq N$, where $C$ is a constant independent of $n$, $\tau$, and $h$.
\end{thm}%--------------------------------------------------------------------------------------------------------------------------
\begin{proof}
Denote $\varepsilon_j^n= u(x_j,t_n)-U_j^n$ and $\varepsilon^n=(\varepsilon_1^n,\varepsilon_2^n,\cdots,\varepsilon_{M-1}^n)^T$.
According to (\ref{fore})-(\ref{matrixfourth}), we obtain
\begin{equation} \label{errormatrix}
(P_\alpha-B_\alpha)\varepsilon^{n+1}=(P_\alpha+B_\alpha)\varepsilon^{n}+\tau R^{n+1/2},
\end{equation}
where  $R^{n+1/2}=(R_1^{n+1/2},R_2^{n+1/2},\cdots,R_{M-1}^{n+1/2})^T$.
The eigenvalues of $ P_\alpha$ are given by
$$\lambda(P_\alpha)_j=1-4b_2^\alpha\sin^2(j\pi /M)>0,\,\,\, j=1,\cdots, M-1.$$
Since $b_2^\alpha\in(1/12,1/6)$, we have $\lambda(P_\alpha)_j\in(1/3,1)$.
 So the matrix $ P_\alpha$ is invertible and %因为v^TP^{T}v>0，所以v^TP^TP^{-1}Pv=v^TP^Tv>0
    positive definite, which means that $P_\alpha^{-1}$ exists and
is also    positive definite. According to Lemma \ref{le3.3}, we know that $(P_\alpha^{-1})^{\frac{1}{2}}$ uniquely  exists and is positive semi-definite. %------------------------
Multiplying $(P_\alpha^{-1})^{\frac{1}{2}}$  and taking the discrete $L_2$ norm on both sides of  (\ref{errormatrix}) imply
\begin{equation*}
 \|((P_\alpha )^{\frac{1}{2}}-(P_\alpha^{-1})^{\frac{1}{2}}B_\alpha)\varepsilon^{n+1}\|\leq
\| ((P_\alpha )^{\frac{1}{2}}+(P_\alpha^{-1})^{\frac{1}{2}}B_\alpha)\varepsilon^{n}\|+\tau \|(P_\alpha^{-1})^{\frac{1}{2}}R^{n+1/2}\|.
\end{equation*}
 In view of Theorem   \ref{ka}, we know that $B_\alpha+B_\alpha^T$ is a negative definite matrix.
 Furthermore,
 \begin{equation}\label{2}
 \begin{array}{l }
\displaystyle~~~~
   ((P_\alpha)^{\frac{1}{2}}-(P_\alpha^{-1})^{\frac{1}{2}}B_\alpha)^T((P_\alpha)^{\frac{1}{2}}-(P_\alpha^{-1})^{\frac{1}{2}}B_\alpha)\\\\
\displaystyle   =P_\alpha-B_\alpha-B_\alpha ^T+B_\alpha ^T P_\alpha^{-1}B_\alpha\geq P_\alpha+B_\alpha ^T P_\alpha^{-1}B_\alpha
   \end{array}
 \end{equation}
 and
 \begin{equation}\label{3}
  \begin{array}{l }
\displaystyle~~~~
   ((P_\alpha)^{\frac{1}{2}}+(P_\alpha^{-1})^{\frac{1}{2}}B_\alpha)^T((P_\alpha)^{\frac{1}{2}}+(P_\alpha^{-1})^{\frac{1}{2}}B_\alpha)\\\\
 \displaystyle  =P_\alpha+B_\alpha+B_\alpha ^T+B_\alpha ^T P_\alpha^{-1}B_\alpha\leq P_\alpha+B_\alpha ^T P_\alpha^{-1}B_\alpha,
  \end{array}
 \end{equation}
where the matrix $A\geq B$ means that $A- B$ is positive semi-definite.
Denote
\begin{equation}\label{5}
 E^n=\sqrt{h(\varepsilon^{n})^T (P_\alpha+B_\alpha ^T P_\alpha^{-1}B_\alpha)\varepsilon^{n}}.
\end{equation}
Since $B_\alpha ^T P_\alpha^{-1}B_\alpha $ is positive definite, we know
\begin{equation}\label{le}
  E^n\geq \sqrt{h(\varepsilon^{n})^T  P_\alpha\varepsilon^{n}}\geq \sqrt{\lambda_{\min} (P_\alpha) }||\varepsilon^{n} ||,
\end{equation}
where $\lambda_{\min}(P_\alpha)$ is the minimum eigenvalue of matrix $P_\alpha $.
Together with (\ref{2}) and (\ref{3}), we have
\begin{equation}\label{con}
 \begin{array}{l }
\displaystyle
  E^{n+1 }-E^n \leq \tau \|(P_\alpha^{-1})^{\frac{1}{2}}R^{n+1/2}\| =\tau\sqrt{h (R^{n+1/2})^T(P_\alpha^{-1}) R^{n+1/2}}\\\\
\displaystyle~~~~  ~~~~ ~~~~ ~~\leq  \tau \sqrt{\lambda_{\max}(P_\alpha^{-1})} \| R^{n+1/2}\|=\frac{\tau }{\sqrt{\lambda_{\min} (P_\alpha)}} \| R^{n+1/2}\|.
   \end{array}
\end{equation}
Summing up (\ref{con}) from $0$ to $n-1$ leads to
\begin{equation}\label{le1}
   E^n\leq \tau\sum_{k=0}^{n-1} \|(P_\alpha^{-1})^{\frac{1}{2}}R^{k+1/2}\| \leq \frac{\tau }{\sqrt{\lambda_{\min} (P_\alpha)}} \sum_{k=0}^{n-1}\| R^{k+1/2}\|.
\end{equation}
Combining (\ref{le}) and (\ref{le1}) and noticing that $|R^{k+1/2}_j|\leq c(\tau^2+h^2)$ for $1\leq j\leq M-1$, we obtain
\begin{equation*}
   \| \varepsilon^n\|\leq\frac{cT}{\lambda_{\min} (P_\alpha)}(\tau^2+h^2)\leq C(\tau^2+h^2).
\end{equation*}
\end{proof}

\section{Quasi-compact scheme for two dimensional space fractional diffusion equation }\label{sec2}%**************************************************************************************
To discuss the quasi-compact scheme in two dimensional case,  we consider the following    space fractional diffusion equation
\begin{equation}\label{equ2}
\left\{
\begin{array}{l ll}
\displaystyle \frac{\partial u(x,t)}{\partial t} =K^x_1\, _aD_x^\alpha u(x,t)+K_2^x\,_xD_b^\alpha u(x,t)\\\\
\displaystyle~~~~ ~~~~ ~~~~  ~~ +K_1^y\, _cD_y^\beta u(x,t)+K_2^y\,_yD_d^\beta u(x,t)+f(x,t),  &(x,y,t)\in \Omega\times(0,T], \\\\
\displaystyle  u(x,y,0)=u_0(x,y),&(x,y)\in \Omega,\\\\
\displaystyle   u(x,y,t)=\phi(x,y,t),\, &(x,y,t)\in \partial\Omega\times(0,T],
\end{array}
\right.
\end{equation}
where $\Omega=(a,b)\times(c,d)$ and the fractional orders $1<\alpha,\beta\leq2$. The diffusion coefficients $K_j^x$ and $K_j^y$ ($j=1,2$) are non-negative and satisfy
$(K_1^j)^2+(K_2^j)^2\neq0$ ($j=x,y$). %In order to   the Riemann-Liouville fractional operator be proper,
The boundary function $\phi$ satisfies the following condition, if $K_1^x\neq 0$, then $\phi(a,y,t)=0$; if $K_1^y\neq 0$, then $\phi(x ,c ,t)=0$;
if $K_2^x\neq 0$, then $\phi(b  ,y,t)=0$; if $K_2^y\neq 0$, then $\phi(x,d,t)=0$.
We assume that the equation (\ref{equ2}) has a unique and sufficiently smooth solution.

\subsection{CN-CWSGD scheme}
Let us denote $x_j=a+jh_x$, $y_s=c+sh_y$, and $t_n=n\tau$ for $0\leq j\leq M_x$, $0\leq s\leq M_y$, and $0\leq n\leq N$, where the space step size $h_x=(b-a)/M_x$, $h_y=(d-c)/M_y$ and  time step size $\tau=T/N$.
Here we take $u_{j,s}^n=u(x_j,y_s,t_n)$ and $ f_{j,s}^{n+1/2}=f(x_j,y_s,t_{n+1/2})$.
%\[\delta_x^2u_{j,s}=u_{i-1,j}-2u_{j,s}+u_{i+1,j},\quad\delta_y^2u_{j,s}=u_{i,j-1}-2u_{j,s}+u_{i,j+1},\quad.\]
%\[\]
 The maximum norm and the discrete $L_2$ norm are defined as
\begin{equation}
\|u\|_\infty=\max\limits_{1\leq j\leq M_x-1,\\ \atop 1\leq s\leq M_y-1}|u_{j,s}|,\quad \|u\|^2= h_xh_y\sum_{j=1}^{M_x-1}\sum_{s=1}^{M_y-1}u_{j,s}^2.
\end{equation}
We still use the Crank-Nicolson technique for the time discretization of equation (\ref{equ2}) and get
\begin{equation}
\begin{array}{lll}
\displaystyle
\frac{u_{j,s}^{n+1}-u_{j,s}^n}{\tau}=\frac{1}{2}\left(K^x_1(_aD_x^\alpha u)_{j,s}^n + K^x_1(_aD_x^\alpha u)_{j,s}^{n+1}  +K^x_2(_xD_b^\alpha u)_{j,s}^n + K^x_2(_xD_b^\alpha u)_{j,s}^{n+1} \right.\\\\
\displaystyle~~~~~~~~~~~~~~~~ ~~\left.+K^y_1(_cD_y^\beta u)_{j,s}^n + K^y_1(_cD_y^\beta u)_{j,s}^{n+1}  +K^y_2(_yD_d^\beta u)_{j,s}^n + K^y_2(_yD_d^\beta u)_{j,s}^{n+1} \right) \\\\\displaystyle~~~~~~~~~~~~~~~~ ~~
+f^{n+1/2}_{j,s}+O(\tau^2).
\end{array}
\end{equation}
In  space, the fourth order quasi-compact discretizations are used   to approximate the Riemann-Liouville fractional derivatives. This implies that
%Let us denote  finite difference operators
%\[P_x u_{j,s}=(1+b_2^\alpha\delta_x)u_{j,s}=b_2^\alpha  u_{i-1,j}+(1-2b_2^\alpha) u_{j,s}  +b_2^\alpha  u_{i+1,j} ,\]
%\[P_y u_{j,s}=(1+b_2^\beta\delta_y)u_{j,s}=b_2^\alpha  u_{i,j-1} +(1-2b_2^\alpha) u_{j,s}  +b_2^\alpha  u_{i,j+1} .\]
%Separating the time layers, we have%\frac{K_1\tau}{2 } \,_{L}D_{h_x}^\alpha u_j^{n} +\frac{K_2\tau}{2 } \,_{R}D_{h_x}^\alpha u_j^{n}
\begin{equation}\label{twodem1}
\begin{array}{l}
\displaystyle
( P_xP_y  -\frac{K^x_1\tau}{2}P_y\,_{L}D_{h_x}^\alpha   -\frac{K^x_2\tau}{2}P_y\,_{R}D_{h_x}^\alpha  %\\\\
%\displaystyle~~~~ ~~~~  ~~~~ ~~~~
-\frac{K^y_1\tau}{2 }P_x\,_{L}D_{h_y}^\alpha  -\frac{K^y_2\tau}{2 }P_x\,_{R}D_{h_y}^\alpha) u_{j,s}^{n+1}\\\\
\displaystyle=(P_xP_y +\frac{K^x_1\tau}{2}P_y\,_{L}D_{h_x}^\alpha   +\frac{K^x_2\tau}{2}P_y\,_{R}D_{h_x}^\alpha
%\\\\\displaystyle~~~~ ~~~~ ~~~~ ~~~~
+\frac{K^y_1\tau}{2 }P_x\,_{L}D_{h_y}^\alpha  +\frac{K^y_2\tau}{2 }P_x\,_{R}D_{h_y}^\alpha) u_{j ,s}^{n}
\\\\\displaystyle~~~~  +\tau P_xP_y f_{j,s}^{n+1/2}+\tau   R^{n+1/2}_{j,s},
\end{array}
\end{equation}
where  $$R^{n+1/2}_{j,s}\leq C(\tau^2+h_x^4+h^4_y).$$
%$w_0^{(\alpha)}=\mu_{1}g_0^{(\alpha)}$， $w_1^{(\alpha)}=\mu_{1}g_1^{(\alpha)}+\mu_{0}g_0^{(\alpha)}$ 且$w_2^{(\alpha)}=\mu_{1}g_k^{(\alpha)}+\mu_{0}g_{k-1}^{(\alpha)}+\mu_{-1}g_{k-2}^{(\alpha)}$。
%Denoting $U_j^n$ as the numerical approximation of $u_j^n$, we obtain the Crank-Nicolson compact scheme for (\ref{equ1})
%\begin{equation}
%\begin{array}{lll}
%\displaystyle~~~~
% P_xP_yU_{j,s}^{n+1}-\frac{K^x_1\tau}{2h^\alpha}P_y\sum_{k=0}^{i+1}w_{k}^{(\alpha)}U_{i-k+1}^{n+1} -\frac{K^x_2\tau}{2h^\alpha}P_y\sum_{k=0}^{M-i+1}w_{k}^{(\alpha)}U_{i+k-1}^{n+1}\\\\
%\displaystyle~~~~ ~~ -\frac{K^y_1\tau}{2h^\alpha}P_x\sum_{k=0}^{i+1}w_{k}^{(\alpha)}U_{i-k+1}^{n+1} -\frac{K^y_2\tau}{2h^\alpha}P_x\sum_{k=0}^{M-i+1}w_{k}^{(\alpha)}U_{i+k-1}^{n+1} \\\\
%\displaystyle
%=P_xP_yU_{j,s}^{n}+\frac{K^x_1\tau}{2h^\alpha}P_y\sum_{k=0}^{i+1}w_{k}^{(\alpha)}U_{i-k+1}^{n} +\frac{K^x_2\tau}{2h^\alpha}P_y\sum_{k=0}^{M-i+1}w_{k}^{(\alpha)}U_{i+k-1}^{n}\\\\
%\displaystyle~~~~ ~~ +\frac{K^y_1\tau}{2h^\alpha}P_x\sum_{k=0}^{i+1}w_{k}^{(\alpha)}U_{i-k+1}^{n} +\frac{K^y_2\tau}{2h^\alpha}P_x\sum_{k=0}^{M-i+1}w_{k}^{(\alpha)}U_{i+k-1}^{n} \\\\
%\displaystyle~~~~+b_2^\alpha   \tau f_{i-1}^{n+1}+(1-2b_2^\alpha  )\tau f_j^{n+1} +b_2^\alpha  \tau  f_{i+1}^{n+1}.
%\end{array}
%\end{equation}
For convenience, we introduce  the following  discrete operator which works for   two variables $x,y$,
\[\delta_x^\alpha u_{j,s}= K^x_1 \,_{L}D_{h_x}^\alpha u_{j,s} + K^x_2 \,_{R}D_{h_x}^\alpha u_{j,s} .\]
%\[\delta_y^\beta u_{j,s}=\frac{K^y_1\tau}{2h^\alpha} \sum_{k=0}^{i+1}w_{k}^{(\alpha)}u_{i-k+1} +\frac{K^y_2\tau}{2h^\alpha} \sum_{k=0}^{M-i+1}w_{k}^{(\alpha)}u_{i+k-1}.\]
Then the equation (\ref{twodem1}) can be rewritten as
\begin{equation}\label{twodem2}
\begin{array}{l}
\displaystyle
( P_xP_y -\frac{\tau}{2}P_y\delta_x^\alpha  -\frac{\tau}{2}P_x\delta_y^\beta) u_{j,s}^{n+1} \\\\
\displaystyle = (P_xP_y +\frac{\tau}{2}P_y\delta_x^\alpha  +\frac{\tau}{2}P_x\delta_y^\beta )u_{j,s}^{n}
+\tau P_xP_y f_{j,s}^{n+1/2}+\tau  R^{n+1/2}_{j,s}.
\end{array}
\end{equation}
Adding the splitting term
\begin{equation}
  \frac{\tau^2}{4}\delta_x^\alpha\delta_y^\beta( u_{j,s}^{n+1}- u_{j,s}^{n})\, (=\tau^3 O(\tau^2+h_x^4+h^4_y)),
\end{equation}
 to the equation (\ref{twodem2}),  we obtain
%The perturbation equation of  the above equation is
\begin{equation}\label{twodem3}
( P_x-\frac{\tau}{2}\delta_x^\alpha)(P_y -  \frac{\tau}{2}\delta_y^\beta) u_{j,s}^{n+1}
= ( P_x+\frac{\tau}{2}\delta_x^\alpha)(P_y + \frac{\tau}{2}\delta_y^\beta)u_{j,s}^{n} +\tau P_xP_y f_{j,s}^{n+1/2}+\tau  R^{n+1/2}_{j,s}.
\end{equation}
Thus the quasi-compact finite difference scheme for (\ref{equ2}) is given by
\begin{equation}\label{twodem22}
\begin{array}{lll}
\displaystyle~~~~
( P_x-\frac{\tau}{2}\delta_x^\alpha)(P_y -  \frac{\tau}{2}\delta_y^\beta) U_{j,s}^{n+1}
= ( P_x+\frac{\tau}{2}\delta_x^\alpha)(P_y + \frac{\tau}{2}\delta_y^\beta)U_{j,s}^{n} +\tau P_xP_y f_{j,s}^{n+1/2}.
\end{array}
\end{equation}
%Comparing (\ref{twodem2}) with (\ref{twodem3}), we know
% Since $U_{j,s}^{n+1}- U_{j,s}^{n}$ is an $O(\tau)$ term, we have that  adding the splitting term implies an $O(\tau^2)$ error component is    contributed   to the truncation error of the Crank-Nicolson scheme(\ref{twodem2}).
 As an efficient way to implementation, we give the following equivalent schemes:%\cite{peaceman1955numerical}
 \begin{itemize}
 %  \item compact Peacemann-Rachord ADI scheme:
%   \begin{equation}\label{pradi}
%    \begin{array}{lll}
%\displaystyle ( P_x-\frac{\tau}{2}\delta_x^\alpha) U_{j,s}^{*}
%= (P_y + \frac{\tau}{2}\delta_y^\beta)U_{j,s}^{n} +\frac{\tau}{2}  P_y f_{j,s}^{n+1/2},\\\\
%\displaystyle  (P_y -  \frac{\tau}{2}\delta_y^\beta)U_{j,s}^{n+1}
%= ( P_x+\frac{\tau}{2}\delta_x^\alpha) U_{j,s}^{*} +\frac{\tau }{2}P_y  f_{j,s}^{n+1/2},
% \end{array}
%   \end{equation}
   \item  quasi-compact Douglas-ADI scheme:
      \begin{equation}\label{d1adi}
    \begin{array}{lll}
\displaystyle ( P_x-\frac{\tau}{2}\delta_x^\alpha) U_{j,s}^{*}= (P_xP_y + \frac{\tau}{2}P_y\delta_x^\alpha+\tau P_x\delta_y^\beta)U_{j,s}^{n} + \tau P_x    P_y f_{j,s}^{n+1/2},\\\\
\displaystyle  (P_y -  \frac{\tau}{2}\delta_y^\beta)U_{j,s}^{n+1}=    U_{j,s}^{*}-\frac{\tau }{2}\delta_y^\beta U _{j,s}^{n};
 \end{array}
   \end{equation}
   \item  quasi-compact D'yakonov-ADI scheme:
       \begin{equation}\label{d2adi}
    \begin{array}{lll}
\displaystyle ( P_x-\frac{\tau}{2}\delta_x^\alpha) U_{j,s}^{*}=( P_x+\frac{\tau}{2}\delta_x^\alpha) (P_y + \frac{\tau}{2}\delta_y^\beta) U_{j,s}^{n} + \tau P_x    P_y f_{j,s}^{n+1/2},\\\\
\displaystyle  (P_y -  \frac{\tau}{2}\delta_y^\beta)U_{j,s}^{n+1}=    U_{j,s}^{*}.
 \end{array}
   \end{equation}
 \end{itemize}

\subsection{Stability and convergence analysis}
The following stability analysis and accuracy analysis indicate that two dimensional CN quasi-compact scheme has fourth order accuracy in space and is unconditionally stable.% Now we give some important lem corresponding to the analyses.
\begin{lem}(\cite{bhatia2009positive})\label{posi}
Let $A$, $B$ be two positive semi-definite matrices, symbolized $A\geq0$, $B\geq0$. Then $A\otimes B\geq0$.
\end{lem}
\begin{lem}(\cite{Laub2005})\label{l3.3}
Let $A\in R^{n\times n}$ have eigenvalues $\{\tilde{\rho}_j\}_{j=1}^n$ and $B\in R^{m\times m}$ have eigenvalues $\{\rho _j\}_{j=1}^m$.
Then the $mn$ eigenvalues of $A\otimes B $ are
\[\tilde{\rho}_1\rho_1,\cdots,\tilde{\rho}_1\rho_m,\,\tilde{\rho}_2\rho_1,\cdots,\tilde{\rho}_2\rho_m,\cdots,\tilde{\rho}_n\rho_1,\cdots,\tilde{\rho}_n\rho_m.\]
\end{lem}
\begin{lem}(\cite{Laub2005}) \label{commute}
Let $A\in R^{m\times n}$, $B\in R^{r\times s}$, $C\in R^{n\times p}$, $D\in R^{s\times t}$. Then
\[(A\otimes B)(C\otimes D)=A C \otimes B D,\]
where $\otimes$ denotes the Kronecker product. Moreover, if $A, B\in R^{n\times n}$, $I$ is a unit matrix of order $n$, then matrices $I\otimes A$
and $B\otimes I$ commute.
\end{lem}
%\begin{lem}(\cite{Laub2005}) \label{commute2}
%If   $A, B\in R^{n\times n}$ are invertible, then we have  $A\otimes B$ is invertible and
% \[(A\otimes B)^{-1}=A^{-1} \otimes B^{-1}.\]
%\end{lem}
\begin{lem}(\cite{Laub2005}) \label{commute2}
Let   $A$ be a $m \times n$ matrix and $B$ a $p\times q$ matrix. We have that the transposition  is distributive over the Kronecker product:
\[ (A\otimes B)^T= A^T\otimes B ^T.\]
\end{lem}
%\begin{thm}\label{ttt}%------------------------------------------------------------------------------------------------------------------------------
%Let the matrix $B_{(\alpha)}$ is defined in (\ref{twob}).   If  $A^{-1}B_{(\alpha)}+(A^{-1}B_{(\alpha)})^T$ are negative, then
%\[\| (I-A^{-1}B_{(\alpha)})^{-1}\|\leq 1,\]
%\[ \| (I-A^{-1}B_{(\alpha)})^{-1}(I-A^{-1}B_{(\alpha)})\|\leq 1.\]
%\end{thm}
%The proof is similar to the Thm \ref{3.2}.%---
%\begin{lem}\label{commute}
%Let $A\in R^{m\times n}$, $B\in R^{r\times s}$, $C\in R^{n\times p}$, $D\in R^{s\times t}$. Then
%\[(A\otimes B)(C\otimes D)=A C \otimes B D,\]
%where $\otimes$ denotes the Kronecker product. Moreover, if $A, B\in R^{N\times n}$, $I$ is a unit matrix of order $n$, then matrices $I\otimes A$
%and $B\otimes I$ commute.
%\end{lem}
%\begin{lem}
%Let $A\in R^{n\times n}$ have eigenvalues $\{\lambda _j\}_{i=1}^n$ and $B\in R^{m\times m}$ have eigenvalues $\{\mu _j\}_{j=1}^m$.
%Then the $mn$ eigenvalues of $A\otimes B $ are
%\[  \lambda_1\mu_1,\cdots,\lambda_1\mu_m,\lambda_2\mu_1,\cdots,\lambda_2\mu_m,\cdots,\lambda_n\mu_1,\cdots,\lambda_n\mu_m.\]
%\end{lem}
\begin{thm}  %-------------------------------------------------------------------------------------------------
For any $1<\alpha , \beta<2$, the finite different scheme (\ref{twodem22})   is unconditionally stable.
\end{thm}%-----------------------------------------------------------------------------------------------------------------------
\begin{proof}
Define the round-off error  as $\epsilon_{j,s}^n=U_{j,s}^n-\tilde{U}_{j,s}^n$.    The    error equation is given by
\begin{equation}\label{222222}
( P_x-\frac{\tau}{2}\delta_x^{\alpha })(P_y -  \frac{\tau}{2}\delta_y^{\beta }) \epsilon_{j,s}^{n+1}
= ( P_x+\frac{\tau}{2}\delta_x^{\alpha })(P_y + \frac{\tau}{2}\delta_y^{\beta })\epsilon_{j,s}^{n} .
\end{equation}
Since the  boundary conditions of the above error equation    are   homogeneous,
  we zero extend the solution of the problem (\ref{222222}) to the whole real plane $R\times R$. It's reasonable to replace the symbols $j+1$ and $M-j+1$ in error equation (\ref{errorequation}) with $\infty$.
Now we have
\begin{equation}\label{extendedtwodem3}
( P_x-\frac{\tau}{2}\delta_x^{\alpha'})(P_y -  \frac{\tau}{2}\delta_y^{\beta'}) \epsilon_{j,s}^{n+1}
= ( P_x+\frac{\tau}{2}\delta_x^{\alpha'})(P_y + \frac{\tau}{2}\delta_y^{\beta'})\epsilon_{j,s}^{n},
\end{equation}
where $$\delta_x^{\alpha'} \epsilon_{j,s}=\frac{K^x_1}{h^\alpha} \sum_{k=0}^{\infty}w_{k}^{(\alpha)} \epsilon_{j-k+1,s}+\frac{K^x_2}{h^\alpha} \sum_{k=0}^{\infty}w_{k}^{(\alpha)}\epsilon_{j+k-1,s},$$ which works for two variables $x,y$.
Let $\epsilon_{j,s}^n=v^ne^{i(j\sigma_1+s\sigma_2)}$, where $i=\sqrt{-1}$, $v^n$ is the amplitude at time level $n$ and
$\sigma_1=2\pi h_x/k_x$ and $\sigma_2=2\pi h_y/k_y$ are the phase angles with wavelength $k_x$ and $k_y$, respectively. Next we just need to prove that the amplification factor $G(\sigma_1,\sigma_2)=v^{n+1}/v^n$ satisfies
the relation $|G(\sigma_1,\sigma_2) |\leq1$ for all $\sigma_1$ and $\sigma_2$ in $[-\pi,\pi]$. In fact, substituting the expressions of $\epsilon_{j,s}^n$ and $\epsilon_{j,s}^{n+1}$ into the equation (\ref{extendedtwodem3}),
%As  boundary conditions of error equation (\ref{twodem3}) are   homogeneous,   the  round-off error can be extended  to the   function   $\epsilon^n(x,y)$ which defined in the whole real axis $R\times R$.
%Then the error equation (\ref{twodem3}) is equivalent to the following scheme
we get the amplification factor
%For linear differential equation with homogeneous boundary  condition, the  error can be expanded in a finite Fourier series, in   $[a,b]$ , as
%\[ \epsilon^n(x)=\sum _{m=1}^{M}v^n(m)e^{i\sigma }\]
%where $M=(b-a)/\Delta x$ and the wavenumber $\sigma=\frac{\pi m}{b-a}$.
%We know that it is enough to consider
%\begin{equation}\label{erroroneterm}
% \epsilon^n(x)=v^n(m)e^{i\sigma }
%\end{equation}
%to get the growth of error.
%Substituting (\ref{erroroneterm}) into (\ref{extendederrorequation}), after some calculation we can get the following equation
% %the equation (\ref{extendederrorequation}) can be written as
%\begin{equation}
%\begin{array}{lll}
%\displaystyle~~~~
%\left((1-2b_2^\alpha  )  +b_2^\alpha   (e^{i\sigma }+e^{-i\sigma } )-\frac{K_1\tau}{2h^\alpha}\sum_{k=0}^{\infty}w_{k}^{(\alpha,\lambda)}e^{-i(k-1)\sigma } \right.\\\\
%\displaystyle~~~~ ~~~~\left.-\frac{K_2\tau}{2h^\alpha}\sum_{k=0}^{\infty}w_{k}^{(\alpha,\lambda)}e^{i(k-1)\sigma } \right)v^{n+1}(\sigma) \\\\
%\displaystyle
%=\left((1-2b_2^\alpha  )  +b_2^\alpha   (e^{i\sigma }+e^{-i\sigma } )+\frac{K_1\tau}{2h^\alpha}\sum_{k=0}^{\infty}w_{k}^{(\alpha,\lambda)}e^{-i(k-1)\sigma } \right. \\\\
% \displaystyle~~~~ ~~~~\left.+\frac{K_2\tau}{2h^\alpha}\sum_{k=0}^{\infty}w_{k}^{(\alpha,\lambda)}e^{i(k-1)\sigma } \right)v^{n}(\sigma).
%\end{array}
%\end{equation}
%It's easy to know that the initial condition can satisfy $\epsilon^0\in L^2(-\infty,+\infty)$.
%So the growth factor  $G(\sigma,\tau)$ is
$$
\begin{array}{l}
\displaystyle G(\sigma_1,\sigma_2)=\frac{ (1-4b_2^\alpha  \sin^2\frac{\sigma_1 }{2}  +\frac{K^x_1\tau}{2h^\alpha}\sum\limits_{k=0}^{\infty}w_{k}^{(\alpha )}e^{-i(k-1)\sigma_1 } +\frac{K^x_2\tau}{2h^\alpha}\sum\limits_{k=0}^{\infty}w_{k}^{(\alpha )}e^{i(k-1)\sigma_1 } ) }
{( 1-4b_2^\alpha  \sin^2\frac{\sigma_1 }{2}-\frac{K^x_1\tau}{2h^\alpha}\sum\limits_{k=0}^{\infty}w_{k}^{(\alpha )}e^{-i(k-1)\sigma_1 }-\frac{K^x_2\tau}{2h^\alpha}\sum\limits_{k=0}^{\infty}w_{k}^{(\alpha )}e^{i(k-1)\sigma_1 }  ) }\\\\
\displaystyle~~~~ ~~~~ ~~~~~ ~~~~~  \cdot\frac{  (1-4b_2^\beta  \sin^2\frac{\sigma_2 }{2}  +\frac{K^y_1\tau}{2h^\beta}\sum\limits_{k=0}^{\infty}w_{k}^{(\beta)}e^{-i(k-1)\sigma_2 } +\frac{K^y_2\tau}{2h^\beta}\sum\limits_{k=0}^{\infty}w_{k}^{(\beta)}e^{i(k-1)\sigma _2} ) }
{ ( 1-4b_2^\beta  \sin^2\frac{\sigma _2 }{2}-\frac{K^y_1\tau}{2h^\beta}\sum\limits_{k=0}^{\infty}w_{k}^{(\beta)}e^{-i(k-1)\sigma _2 }-\frac{K^y_2\tau}{2h^\beta}\sum\limits_{k=0}^{\infty}w_{k}^{(\beta)}e^{i(k-1)\sigma_2 }  )}\\\\
\displaystyle~~~~ ~~~~ ~~~~   = \frac{Q_1(\sigma_1,\alpha)  +Q_2(\sigma_1,\alpha)  }
{ Q_1(\sigma_1,\alpha) - Q_2(\sigma_1,\alpha)   }\cdot \frac{Q_1(\sigma_2,\beta)  +Q_2(\sigma_2,\beta)  }
{ Q_1(\sigma_2,\beta) - Q_2(\sigma_2,\beta)   }\\\\
\displaystyle~~~~~~~~ ~~~~  =  v(\sigma_1,\alpha)  \cdot  v(\sigma_2,\beta) ,
\end{array}
$$
where $Q_1(\sigma_1,\alpha)= 1-4b_2^\alpha  \sin^2\frac{\sigma_1 }{2}$ and $Q_2(\sigma_1,\alpha)=\frac{K^x_1\tau}{2h^\alpha}\sum\limits_{k=0}^{\infty}w_{k}^{(\alpha)}e^{-i(k-1)\sigma_1 } +\frac{K^x_2\tau}{2h^\alpha}\sum\limits_{k=0}^{\infty}w_{k}^{(\alpha)}e^{i(k-1)\sigma _1}$, which work  for two pairs of  variables $(\sigma_1,\alpha)$ and $(\sigma_2,\beta)$.
%As we all know, $|G(\sigma,\tau)|=|G_x|\cdot|G_y|$.
According to the analysis of Theorem \ref{stability1d}, we know that $|v(\sigma_1,\alpha)|\leq1$ and  $| v(\sigma_2,\beta)|\leq1$ hold  for any $\alpha,\beta\in(1,2)$. %and $\in(1,2)$.
Then $$|G(\sigma_1,\sigma_2)|= |v(\sigma_1,\alpha)| \cdot| v(\sigma_2,\beta)|\leq 1.$$
So the C-N quasi-compact scheme is unconditionally stable.

\end{proof}
\begin{thm}  %-------------------------------------------------------------------------------------------------
Let $u(x_j,y_s,t_n)$ be the exact solution of equation (\ref{equ2}), and $U_{j,s}^n$ the solution of the given finite difference scheme (\ref{twodem22}). Then we have
\[\|u(x_j,y_s,t_n)-U_{j,s}^n\|\leq C(\tau^2+h_x^4+h^4_y),\]
for all $1\leq n\leq N$, where $C$ is a constant independent of $\tau$, $h_x$, and $h_y$.
\end{thm}%-----------------------------------------------------------------------------------------------------------------------
\begin{proof}
Denote  $\varepsilon_{j,s}^n= u(x_j,y_s,t_n)-U_{j,s}^n$,
  and
\[P_{(\alpha)}=I_\beta\otimes P_\alpha,\quad P_{(\beta)}=P_\beta\otimes I_\alpha,\]
\[(P_{(\alpha)})^\frac{1}{2}=I_\beta\otimes (P_\alpha)^\frac{1}{2},\quad (P_{(\beta)})^\frac{1}{2}=(P_\beta)^\frac{1}{2}\otimes I_\alpha,\]
\begin{equation}\label{twob}
B_{(\alpha)}=\frac{K_1^x\tau}{2h^\alpha_x}I_\beta\otimes A_\alpha+\frac{K_2^x\tau}{2h^\alpha_x}I_\beta\otimes A_\alpha^T,\quad
B_{(\beta)}=\frac{K_1^y\tau}{2h^\beta_y} A_\beta\otimes I_\alpha+\frac{K_2^y\tau}{2h^\beta_y}A_\beta^T\otimes I_\alpha,
\end{equation}
where $A_\alpha$ and $A_\beta$ are defined in (\ref{mA}) corresponding to $\alpha$ and $\beta$.
In view of (\ref{twodem3})-(\ref{twodem22}), we  obtain
\begin{equation}\label{777}
   (P_{(\alpha)}-B_{(\alpha)})(P_{(\beta)}-B_{(\beta)})\varepsilon^{n+1}=(P_{(\alpha)}+B_{(\alpha)})(P_{(\beta)}+B_{(\beta)})\varepsilon^{n}+\tau R^{n+1/2},
\end{equation}
where $$\varepsilon=(\varepsilon_{1,1},\varepsilon_{2,1},\cdots,\varepsilon_{M_x-1,1},\varepsilon_{1,2},
\varepsilon_{2,2},\cdots,\varepsilon_{M_x-1,2},\varepsilon_{1,M_y-1},\varepsilon_{2,M_y-1},\cdots,\varepsilon_{M_x-1,M_y-1})^T.$$
%According to Lem \ref{le3.3}, we know that $(P_{(\alpha)}^{-1})^{\frac{1}{2}}$ and  $(P_{(\beta)}^{-1})^{\frac{1}{2}}$ uniquely  exist  and are positive semi-definite. %------------------------
Multiplying $(P_{(\alpha)}^{-1})^{\frac{1}{2}} (P_{(\beta)}^{-1})^{\frac{1}{2}}$  and taking the discrete $L_2$ norm on both sides of   equation (\ref{777}) imply
\begin{equation}\label{5555}
 \begin{array}{l}
\displaystyle
\|(P_{(\alpha)}^{-1})^{\frac{1}{2}} (P_{(\beta)}^{-1})^{\frac{1}{2}} (P_{(\alpha)}-B_{(\alpha)})(P_{(\beta)}-B_{(\beta)})\varepsilon^{n+1}\|\\\\
\displaystyle\leq\|(P_{(\alpha)}^{-1})^{\frac{1}{2}} (P_{(\beta)}^{-1})^{\frac{1}{2}}(P_{(\alpha)}+B_{(\alpha)})(P_{(\beta)}+B_{(\beta)})\varepsilon^{n}\|+\tau \|(P_{(\alpha)}^{-1})^{\frac{1}{2}} (P_{(\beta)}^{-1})^{\frac{1}{2}}R^{n+1/2}\|.
\end{array}
\end{equation}
%By Lem \ref{l3.3},  we see that $ P_{(\alpha)}$ and $ P_{(\beta)}$ are   positive definite. % , which means $ P_{(\alpha)}^{-1}$ and $ P_{(\beta)}^{-1}$
%are also  positive definite. %According to Thm \ref{ka}, the matrix $B$ is negative definite.
Using Lemmas \ref{commute} and  \ref{commute2}, it is easy to check that the matrix
$(P_{(\beta)}^{-1})^{\frac{1}{2}}$  can commute with $(P_{(\alpha)}^{-1})^{\frac{1}{2}} $ and $ P_{(\alpha)}\pm B^T_{(\alpha)}$, i.e.,
 \[(P_{(\beta)}^{-1})^{\frac{1}{2}}(P_{(\alpha)}^{-1})^{\frac{1}{2}}= (P_{(\alpha)}^{-1})^{\frac{1}{2}} (P_{(\beta)}^{-1})^{\frac{1}{2}}
 =(P_{ \beta}^{-1})^{\frac{1}{2}}\otimes(P_{ \alpha }^{-1})^{\frac{1}{2}},\]
\[(P_{(\beta)}^{-1})^{\frac{1}{2}}(P_{(\alpha)}\pm B^T_{(\alpha)})=(P_{(\alpha)}\pm B^T_{(\alpha)})(P_{(\beta)}^{-1})^{\frac{1}{2}}= (P_{ \beta }^{-1})^{\frac{1}{2}}\otimes \left(P_\alpha\pm\frac{K_1^x\tau}{2h^\alpha_x}  A^T_\alpha\pm\frac{K_2^x\tau}{2h^\alpha_x}  A_\alpha\right).\]
After some similar  calculations, we also  get that $P_{(\beta)}-B_{(\beta)}$ commutes with $P_{(\alpha)}-B_{(\alpha)}$, $(P_{(\alpha)}^{-1})^{\frac{1}{2}} $, and $ P_{(\alpha)}-B^T_{(\alpha)}$; and $P_{(\beta)}+B_{(\beta)}$    commutes with $P_{(\alpha)}+B_{(\alpha)}$, $(P_{(\alpha)}^{-1})^{\frac{1}{2}} $, and $ P_{(\alpha)}+B^T_{(\alpha)}$.
In view of Theorem  \ref{ka}, we know that $B_{ \alpha }+B_{ \alpha }^T$ and $B_{ \beta }+B_{ \beta }^T$ are negative definite matrixes.
Together with Lemma \ref{l3.3}, it yields that $B_{(\alpha)}+B_{(\alpha)}^T$ and $B_{(\beta)}+B_{(\beta)}^T$ are also negative definite matrixes.
 Using Lemma \ref{posi}, there exist
 \begin{equation}\label{2222}
 \begin{array}{l}
\displaystyle
   ((P_{(\alpha)}^{-1})^{\frac{1}{2}} (P_{(\beta)}^{-1})^{\frac{1}{2}} (P_{(\alpha)}-B_{(\alpha)})(P_{(\beta)}-B_{(\beta)}))^T (P_{(\alpha)}^{-1})^{\frac{1}{2}} (P_{(\beta)}^{-1})^{\frac{1}{2}} (P_{(\alpha)}-B_{(\alpha)})(P_{(\beta)}-B_{(\beta)}) \\\\
 \displaystyle \geq (P_{(\beta)}+B_{(\beta)} ^T P_{(\beta)}^{-1}B_{(\beta)})(P_{(\alpha)}+B_{(\alpha)} ^T P_{(\alpha)}^{-1}B_{(\alpha)})+(B_{(\beta)}+B_{(\beta)}^T)(B_{(\alpha)}+B_{(\alpha)}^T)
  \end{array}
 \end{equation}
 and
 \begin{equation}\label{3333}
  \begin{array}{l}
 \displaystyle
   ((P_{(\alpha)}^{-1})^{\frac{1}{2}} (P_{(\beta)}^{-1})^{\frac{1}{2}} (P_{(\alpha)}+B_{(\alpha)})(P_{(\beta)}+B_{(\beta)}))^T (P_{(\alpha)}^{-1})^{\frac{1}{2}} (P_{(\beta)}^{-1})^{\frac{1}{2}} (P_{(\alpha)}+B_{(\alpha)})(P_{(\beta)}+B_{(\beta)}) \\\\
 \displaystyle \leq (P_{(\beta)}+B_{(\beta)} ^T P_{(\beta)}^{-1}B_{(\beta)})(P_{(\alpha)}+B_{(\alpha)} ^T P_{(\alpha)}^{-1}B_{(\alpha)})+(B_{(\beta)}+B_{(\beta)}^T)(B_{(\alpha)}+B_{(\alpha)}^T),
  \end{array}
 \end{equation}
where the matrix $A\geq B$ means that $A- B$ is positive semi-definite.
Denoting $ E^n
  =\sqrt{h(\varepsilon^{n})^T ((P_{(\beta)}+B_{(\beta)} ^T P_{(\beta)}^{-1}B_{(\beta)})(P_{(\alpha)}+B_{(\alpha)} ^T P_{(\alpha)}^{-1}B_{(\alpha)})+(B_{(\beta)}+B_{(\beta)}^T)(B_{(\alpha)}+B_{(\alpha)}^T))\varepsilon^{n}}$,
%\begin{equation}\label{5}
% \begin{array}{l}
% \displaystyle
% E^n
%  =\sqrt{h(\varepsilon^{n})^T ((P_{(\beta)}+B_{(\beta)} ^T P_{(\beta)}^{-1}B_{(\beta)})(P_{(\alpha)}+B_{(\alpha)} ^T P_{(\alpha)}^{-1}B_{(\alpha)})+(B_{(\beta)}+B_{(\beta)}^T)(B_{(\alpha)}+B_{(\alpha)}^T))\varepsilon^{n}}.
% \end{array}
% \end{equation}
%Form $B_\alpha ^T P_\alpha^{-1}B_\alpha \epsilon^{n}$ is positive  definite,
we have
\begin{equation}\label{lee}
  E^n\geq \sqrt{h(\varepsilon^{n})^T  (P_{(\alpha)}) (P_{(\beta)})\varepsilon^{n}}\geq \sqrt{\lambda_{\min} (P_\alpha)\lambda_{\min} (P_\beta) }||\varepsilon^{n} ||,
\end{equation}
where $\lambda_{\min}(P_\alpha)$ and $\lambda_{\min}(P_\beta)$ are the minimum eigenvalues of matrix $P_\alpha $ and $P_\beta $, respectively.
Together with (\ref{2222}) and (\ref{3333}), we have
\begin{equation*}
 \begin{array}{l}
 \displaystyle~~~~
  E^{n+1 } \leq  E^0+\tau \sum_{k=0}^{n }\| (P_{(\alpha)}^{-1})^{\frac{1}{2}} (P_{(\beta)}^{-1})^{\frac{1}{2}} R^{n+1/2}\|
  \leq \tau \sum_{k=0}^{n } \sqrt{\lambda_{\max} ( P_{(\alpha)}^{-1}    P_{(\beta)}^{-1}   ) }\| R^{n+1/2}\|\\\\
\displaystyle ~~~~  ~~~~ ~~~~  = \frac{\tau }{\sqrt{\lambda_{\min} (P_\alpha)\lambda_{\min} (P_\beta)}}  \sum_{k=0}^{n }\| R^{n+1/2}\|.
 % =\frac{\tau }{\sqrt{\lambda_{\min} (P_\alpha)\lambda_{\min} (P_\beta)}} \| R^{n+1/2}\|.
\end{array}
\end{equation*}
%Summing up (\ref{con}) for all $0\leq k\leq n-1$ show that
%\begin{equation}\label{le1}
%   E^n\leq \tau\sum_{k=0}^{n-1} \|(P_\alpha^{-1})^{\frac{1}{2}}R^{k+1/2}\| \leq.
%\end{equation}
Using (\ref{lee}) and noticing that $|R^{k+1/2}_{j,s}|\leq c(\tau^2+h_x^2+h_y^2)$ for $1\leq j\leq M_x-1$ and $1\leq s\leq M_y-1$, we obtain
\begin{equation*}
   \| \varepsilon^n\|\leq\frac{cT}{\lambda_{\min} (P_\alpha)\lambda_{\min} (P_\beta)}(\tau^2+h_x^2+h_y^2)\leq C(\tau^2+h_x^2+h_y^2).
\end{equation*}
\end{proof}
\section{Extending quasi-compact discretizations and schemes to tempered space fractional derivative and equation}
%Now from the Taylor's expansions of the shift tempered Gr\"{u}nwald-Letnikov operator and similar to get the CWSGD operator given in Sec. 2  , we derive  fourth and fifth compact
%difference operators for Riemann-Liouville  tempered fractional derivative,
%whose defn is introduce as follow.
This section focuses on developing the high order quasi-compact schemes of tempered fractional differential equation with Dirichlet boundary condition.
 We begin with the definitions of $\alpha$-th order left and right Riemann-Liouville tempered fractional derivatives.
\begin{defn}(\cite{li2014high})
If the function $u(x)$ is defined in finite interval $[a,b]$ and regular enough, then
%any subinterval of $[-\infty,+\infty]$. Then
  for any $\lambda\geq0$
 the  $\alpha$-th order left and right Riemann-Liouville tempered fractional derivatives are, respectively, defined as
\begin{equation}%\label{lefrie}
_a D_x^{\alpha,\lambda}u(x)=e^{-\lambda x}\,_a D_x^{\alpha}(e^{\lambda x}u(x))=\frac{e^{-\lambda x}}{\Gamma(n-\alpha)}\frac{d^n}{dx^n} \int^{x}_{a}(x-s)^{n-\alpha-1}e^{\lambda s}u(s) ds
\end{equation}
and
\begin{equation}
_x D_{b}^{\alpha,\lambda}u(x)=e^{\lambda x}\,_x D_{b}^{\alpha}(e^{-\lambda x}u(x))=\frac{(-1)^ne^{\lambda x}}{\Gamma(n-\alpha)}\frac{d^n}{dx^n} \int^{b}_{x}(s -x)^{n-\alpha-1}e^{-\lambda s}u(s) ds ,
\end{equation}
where $ n-1<\alpha<n$.
Moreover, if $\lambda=0$, then the derivatives $_{a} D_x^{\alpha,\lambda}u(x)$ and $_x D_{b}^{\alpha,\lambda}u(x)$ reduce to the derivatives $_{a} D_x^{\alpha}u(x)$ and $_x D_{b}^{\alpha}u(x)$ defined in Definition \ref{drl}.
\end{defn}

For getting the stable scheme, we introduce a  shifted Gr\"{u}nwald-Letnikov operator to approximate the left tempered Riemann-Liouville
fractional derivative with first order accuracy.
%And in []
\begin{lem}[\cite{li2014high}]\label{lem14}%----------------------------------------------------------------------------------
 Let $1<\alpha<2$,   $u\in C^{n+3}(R)$ such that $D^ku(x)\in L^1(R)$, $k=0,1,  \cdots,n+3$. For any integer $p$ and   $\lambda\geq0$
define the left shifted tempered Gr\"{u}nwald-Letnikov operator by
\begin{equation}\label{delf4}
\Delta_p^{\alpha,\lambda} u(x):= \frac{1}{h^\alpha}\sum\limits_{k=0}^\infty g_k^{(\alpha)}e^{-(k-p)\lambda h}u(x-(k-p)h).
\end{equation}
Then   we have
\begin{equation}\label{scheme04}
\Delta_p^{\alpha,\lambda} u(x)=\,_{-\infty}D^{\alpha,\lambda}_x u(x)+\sum\limits_{l=1}^{n-1}a_{p, l}^{\alpha} \, _{-\infty}D_x^{\alpha+l,\lambda}u(x)h^l+O(h^n)
\end{equation}
uniformly in $x\in R$, where the weights $a_{p, l}^\alpha $ are the same as Lemma \ref{lem1}.
\end{lem}%-------------------------------------------------------------------------------------

To approximate the right Riemann-Liouville  tempered fractional derivative $\,_{x}D^{\alpha,\lambda}_{+\infty} u(x) $,  the right shifted tempered Gr\"{u}nwald-Letnikov operator is defined as
$
\Lambda_p^{\alpha,\lambda} f(x):= \frac{1}{h^\alpha}\sum\limits_{k=0}^\infty g_k^{(\alpha)}e^{-(k-p)\lambda h}u(x+(k-p)h)
$.
If the function $u(x)$ is defined on the bounded  interval  $[a,b]$, then the shifted tempered Gr\"{u}nwald-Letnikov  formulae approximating the  tempered fractional derivative at point $x$ are written as
\begin{equation}\label{5.7}
\begin{array}{lll}
\displaystyle\tilde{\Delta}_p^{\alpha,\lambda} u(x)= \frac{1}{h^\alpha}\sum\limits_{k=0}^{[\frac{x-a }{h}]+p} g_k^{(\alpha)}e^{-(k-p)h\lambda}u(x-(k-p)h),\\\\
\displaystyle\tilde{\Lambda}_p^{\alpha,\lambda} u(x)= \frac{1}{h^\alpha}\sum\limits_{k=0}^{[\frac{b -x}{h}]+p} g_k^{(\alpha)}e^{-(k-p)h\lambda}u(x+(k-p)h).
\end{array}
\end{equation}
%$-e^{ph\lambda}\frac{(1-e^{-h\lambda})^\alpha}{h^\alpha}u(x)-e^{ph\lambda}\frac{(1-e^{-h\lambda})^\alpha}{h^\alpha}u(x)$
Next we establish  some suitable high order finite difference discretizations to approximate the tempered fractional derivative.

 \subsection{Quasi-compact discretizations to the tempered Riemann-Liouville space fractional derivative}

Now from the Taylor's expansions of the shifted tempered Gr\"{u}nwald-Letnikov operator, similar to get the CWSGD operator given in Sec. 2, we derive the fourth and fifth order quasi-compact
difference operators for Riemann-Liouville tempered fractional derivative.
\subsubsection{Fourth order quasi-compact approximation to the tempered Riemann-Liouville fractional derivative}
%First we show the fourth order combined quasi-compact approximation.

\begin{thm}\label{4.1}%-----------------------------------------------------------------------------------------------------------------
Let $u(x)\in C^7(R)$ and  all the derivatives of $u(x)$  up to order 7 belong to $L_1(R)$.  Then the following
quasi-compact approximation has fourth order accuracy, i.e.,
% Then the combined quasi-compact approximating to the left Riemann-Liouville tempered  fractional derivative have fourth order accuracy.
\begin{equation}\label{lrlt4}
\begin{array}{llll}
\displaystyle~~~~P_x^\lambda \,_{-\infty}D^{\alpha,\lambda}_x u(x)
 =\mu_1 \Delta_1^{\alpha,\lambda} u(x)+ \mu_0 \Delta_0^{\alpha,\lambda} u(x)+ \mu_{-1} \Delta_{-1}^{\alpha,\lambda} u(x) +O(h^4),
\end{array}
\end{equation}
where $P_x^\lambda u(x)=u(x)+h^2b_2^\alpha e^{-\lambda x}\delta_x^2(e^{\lambda x}u(x))$ and the coefficients $ b_2^\alpha  $,   $\mu_1$, $\mu_0$ and $\mu_{-1}$ are   given by (\ref{coeff4}).
\end{thm}

Note that, by  Lemma \ref{lem14}, the following equation holds,
%%By the
% under the assumptions of the above thm,
 %we know for any  coefficients    $\mu_1 $, $\mu_0 $ and $\mu_{-1} $
%  the following equality holds.
\begin{equation}\label{schm114}
\begin{array}{l}
\displaystyle\mu_1 \Delta_1^{\alpha,\lambda}u(x)+ \mu_0 \Delta_0^{\alpha,\lambda} u(x)+\mu_{-1} \Delta_{-1}^{\alpha,\lambda} u(x)\\\\
 \displaystyle =\,_{-\infty}D^{\alpha,\lambda}_x u(x)+b_2^\alpha   \,_{-\infty}D_x^{\alpha+2,\lambda}u(x)h^2+O(h^4)\\\\
  \displaystyle  =(1+h^2b_2^\alpha \,_{-\infty}D_x^{2,\lambda})\,_{-\infty}D^{\alpha,\lambda}_x u(x)  +O(h^4)\\\\
   \displaystyle  =(1+h^2b_2^\alpha e^{-\lambda x}  \frac{\partial ^2}{\partial x^2}e^{\lambda x})\,_{-\infty}D^{\alpha,\lambda}_x u(x)  +O(h^4)\\\\
   \displaystyle  =P^\lambda_x\,_{-\infty}D^{\alpha,\lambda}_x u(x)  +O(h^4).
\end{array}
\end{equation}
Then we get (\ref{lrlt4}).
%where
%$P^\lambda_x u(x)=
%=b_2^\alpha e^{-\lambda h}u(x-h)+(1-2 b_2^\alpha )u(x)+b_2^\alpha e^{\lambda h}u(x+h)$.
Since $\delta_x^2 u=\frac{\partial^2}{\partial x^2}u+O(h^2)$, we know for any function $u$
$$
P^\lambda_x u=(1+h^2b_2^\alpha \,_{-\infty}D_x^{2,\lambda})u+O(h^4).
$$
In a similar way, we obtain quasi-compact approximation of the right Riemann-Liouville tempered fractional derivative:
\begin{equation}
P_x^\lambda\,_x D^{\alpha,\lambda}_{+\infty} u(x)
 =\mu_1 \Lambda_1^{\alpha,\lambda} u(x_j)+\mu_0 \Lambda_0^{\alpha,\lambda} u(x_j)+ \mu_{-1} \Lambda_{-1}^{\alpha,\lambda} u(x_j) +O(h^4).
\end{equation}
For $u(x)$ defined on a bounded interval, supposing its zero extension to $R$ satisfies the assumptions of Theorem \ref{4.1}, the following approximations hold:
\begin{equation}
P_x\,{_{a}D}^{\alpha,\lambda}_x u(x)=\mu_1 \tilde{\Delta}_1^{\alpha,\lambda} u(x)+ \mu_0 \tilde{\Delta}_0^{\alpha,\lambda} u(x)+ \mu_{-1} \tilde{\Delta}_{-1}^{\alpha,\lambda} u(x)+O(h^4)
\end{equation}
and
\begin{equation}
P_x\,{_{x}D}^{\alpha,\lambda}_b u(x)=\mu_1 \tilde{\Lambda}_1^{\alpha,\lambda} u(x)+ \mu_0 \tilde{\Lambda}_0^{\alpha,\lambda} u(x)+ \mu_{-1} \tilde{\Lambda}_{-1}^{\alpha,\lambda} u(x)+O(h^4).
\end{equation}
Next we give an example to verify the efficiency and convergence order of the above statement.
 \begin{example}\label{ex3}
Consider the steady state tempered fractional diffusion problem
\[ _0D_x^{\alpha,\lambda }u(x)=\frac{720 e^{-\lambda x} x^{6 -  \alpha}}{\Gamma(7 - \alpha)} ,\quad  x\in(0,1),\]
with the boundary conditions $u(0)=0$ and $u(1)=e^{-\lambda  } $, and $\alpha \in (1,2)$. The exact solution is given by $u(x)=e^{-\lambda x}x^{6}$.
\end{example}

Let us denote $u$ and $U$ as the exact solution and approximate value, respectively. In Table \ref{tabtwo111}, we show that the proposed approxiamtion in this subsection has fourth order accuracy in $L_\infty$ norm and $L_2$ norm.
\begin{table}[htbp]
\centering\small
\caption{Numerical errors and  convergence rates in $L_\infty$ norm and $L_2$ norm  of scheme (\ref{lrlt4}) to solve Example \ref{ex3}, where $U$ denotes  the numerical solution,   $h_x$ is space step size and  $\lambda=1.5$.}
\vspace{1em}{\begin{tabular}{@{}cccccc@{}} \hline
 $\alpha$ & $h_x$ & $\|u-U\|_2$ &rate & $\|u-U\|_\infty$ &rate  \\
 \hline
1.1& $1/8$   & $    3.8735e-04 $ &   $         $ & $   8.7474e-04 $ &   $ $ \\
& $1/16 $ &  $   1.8576e-05  $ &   $   4.3821 $  &  $    4.6954e-05$ &   $  4.2195  $ \\
& $ 1/32$   &  $    1.0159e-06  $  &   $   4.1926  $   & $    2.6950e-06$  &    $ 4.1229  $  \\
& $1/ 64$   &  $  6.0438e-08   $  &   $   4.0712  $   &  $  1.6005e-07   $  &    $   4.0737  $ \\
& $1/128$   &  $    3.6901e-09 $  &   $   4.0337  $  &  $   9.4537e-09 $  &   $  4.0815   $ \\
\hline
1.9&$1/8$ &  $  6.2019e-05    $ &   $         $ & $     8.8032e-05   $ &   $      $ \\
&$1/16 $&     $    3.8991e-06 $ &   $   3.9915 $  &  $    5.6382e-06$ &   $   3.9647 $ \\
&$1/ 32$ &     $  2.4425e-07  $  &  $   3.9967 $   &  $  3.5328e-07 $  &    $  3.9964  $ \\
&$1/ 64$ &     $   1.5281e-08   $  &  $   3.9985 $   &  $    2.2104e-08 $  &    $    3.9984$ \\
&$1/128$ &     $   9.5548e-10  $  &  $ 3.9994   $  &  $       1.3822e-09$  &   $    3.9993  $ \\
\hline
\end{tabular}}\label{tabtwo111}
\end{table}

\subsubsection{Fifth order quasi-compact approximation to the tempered Riemann-Liouville fractional derivative}
\begin{thm}
Let $u(x)\in C^8(R)$. Then the quasi-compact approximations corresponding to the left Riemann-Liouville tempered fractional derivative have fifth order accuracy,
\begin{equation}\label{fif5}
\displaystyle P_x^{\lambda,5}\,_{-\infty}D^{\alpha,\lambda }_x u(x)
  =\mu_1 \Delta_1^{\alpha,\lambda } u(x)+ \mu_0 \Delta_0^{\alpha,\lambda } u(x)+ \mu_{-1} \Delta_{-1}^{\alpha,\lambda } u(x) +O(h^5),
\end{equation}
where the operator $P_x^{\lambda,5}u(x) = m e^{-\lambda h}u(x-h) +u(x)+  n e^{\lambda h}u(x+h) $ and the coefficients $ m$, $ n$, $\mu_1$, $\mu_0$ and $\mu_{-1}$ satisfy (\ref{fifcoef}).
\end{thm}

Similar to the discussions in Subsection 2.2, we show three equalities
 \begin{equation} \label{2.2.15}
\Delta_p^{\alpha,\lambda } u(x)=\,_{-\infty}D^{\alpha,\lambda }_x u(x)+\sum\limits_{l=1}^{4}a_{p, l}^\alpha\, \, _{-\infty}D_x^{\alpha+l,\lambda}u(x)h^l+O(h^5),\quad p=1,0,-1.
\end{equation}
 In view of  the Taylor  expansion we know
 \begin{equation}\label{2.2.25}
\begin{array}{llll}
\displaystyle  _{-\infty}D^\alpha _xe^{\lambda (x-h) }u (x-h)  =\,_{-\infty}D^\alpha_x e^{\lambda x }u(x )+(-1)^l\sum\limits_{l=1}^{4}\frac{1}{l!}\,_{-\infty}D^{\alpha+l}_xe^{\lambda x } u(x )h^l+O(h^5),\\\\
 \displaystyle  _{-\infty}D^\alpha_xe^{\lambda (x+h)}u(x+h) =\,_{-\infty}D^\alpha_x e^{\lambda x }u(x )+\sum\limits_{l=1}^{4}\frac{1}{l!}\,_{-\infty}D^{\alpha+l}_x e^{\lambda x }u(x )h^l+O(h^5).
\end{array}
\end{equation}
As $e^{\lambda x} \, _{-\infty}D^{\alpha,\lambda }_x u(x )=\,_{-\infty}D^\alpha _xe^{\lambda x }u(x )$,
multiplying $e^{-\lambda x }$ in equations of (\ref{2.2.25}) we obtain
 \begin{equation}\label{2.2.35}
\begin{array}{llll}
\displaystyle e^{-\lambda h}\, _{-\infty}D^{\alpha,\lambda }_xu(x-h) =\,_{-\infty}D^{\alpha,\lambda }_x u(x )+(-1)^l\sum\limits_{l=1}^{4}\frac{1}{l!}\,_{-\infty}D^{\alpha+l,\lambda}_x u(x )h^l+O(h^5),\\\\
 \displaystyle e^{\lambda h}\, _{-\infty}D^{\alpha,\lambda }_xu(x+h) =\,_{-\infty}D^{\alpha,\lambda }_x u(x )+\sum\limits_{l=1}^{4}\frac{1}{l!}\,_{-\infty}D^{\alpha+l,\lambda}_x u(x )h^l+O(h^5).
\end{array}
\end{equation}
So in order to get the fifth order approximation, combining (\ref{2.2.15}) and (\ref{2.2.35}), we just need to eliminate the low order terms corresponding to
$h^k(k=1,2,3,4)$. Then we get the equation (\ref{fif5}).
To show the efficiency of the proposed  approximation in this subsection, we numerically solve the Example \ref{extwo2} and present the numerical results in Table \ref{tabtwo2}, where $u$ and $U$ denote the exact solution and approximate value, respectively. Obviously, the approximations have fifth order accuracy
which verify the theoretical analysis.
 \begin{example}\label{extwo2}
 Here we also
consider the steady state tempered fractional diffusion problem
\[ _0D_x^{\alpha,\lambda} u(x)=\frac{720e^{-\lambda x} x^{6 -  \alpha}}{\Gamma(7 - \alpha)} ,\quad x\in(0,1)\]
with the boundary conditions $u(0)=0$ and $u(1)=e^{-\lambda} $, and $\alpha \in (1,2)$. The exact solution is $u(x)=e^{-\lambda x}x^{6}$.
\end{example}
\begin{table}[htbp]
\centering\small
\caption{Numerical errors and   convergence rates in $L_\infty$ norm and $L_2$ norm of scheme (\ref{fif5}) to solve Example \ref{extwo2}, where $U$ denotes  the numerical solution, $h_x$ is space step size and  $\lambda=1.5$.}
\vspace{1em}{\begin{tabular}{@{}cccccc@{}} \hline
 $\alpha$ & $h_x$ & $\|u-U\|_2$ &rate & $\|u-U\|_\infty$ &rate  \\
 \hline
1.1& $1/8$   & $      8.2144e-06  $ &   $         $ & $    1.4011e-05$ &   $ $  \\
& $1/16 $  &  $       2.4016e-07  $ &   $     5.0961  $  &  $     4.1068e-07  $ &   $  5.0924   $  \\
& $1/ 32$  &  $     7.3703e-09  $  &   $    5.0261  $   & $     1.2489e-08   $  &    $    5.0392    $  \\
& $1/ 64$   &  $    2.2851e-10  $  &   $  5.0114   $   &  $    3.8494e-10    $  &    $   5.0199    $   \\
& $1/128$   &  $    7.1140e-12  $  &   $    5.0054  $  &  $      1.1945e-11  $  &   $    5.0101  $  \\
\hline
1.5& $1/8$     & $  3.1463e-06  $ &   $          $ & $    5.3572e-06  $  &  $  $\\
& $1/16 $   &  $   9.7972e-08 $ &   $   5.0051   $    &  $    1.6423e-07 $ &   $   5.0276   $ \\
& $1/ 32$    &  $   3.1300e-09  $  &  $   4.9681  $    &  $   5.1523e-09 $  &   $  4.9944    $ \\
& $1/ 64$     &  $   9.9944e-11   $  &   $   4.9689   $   &  $     1.6239e-10 $  &    $   4.9876    $  \\
& $1/128$     &  $     3.1783e-12  $  &   $   4.9748   $   &  $  5.1120e-12     $  &   $   4.9895  $\\
\hline
\end{tabular}}\label{tabtwo2}
\end{table}

\subsection{Quasi-compact scheme for tempered space fractional diffusion equation}
In this subsection, we present the numerical scheme of the variant of  space fractional diffusion equation
whose space fractional derivatives are replaced by the tempered fractional derivatives:
%Based on the hight order approximations to the  Riemann-Liouville tempered space  fractional derivatives,
%we develop some hight order combined  quasi-compact Crank-Nicolson schemes of  the two-side space fractional diffusion equations.
%%\subsection{}
%We consider a initial boundary value problem of the space advection diffusion equation
\begin{equation}\label{equ1la}
\left\{
\begin{array}{lll}
\displaystyle \frac{\partial u(x,t)}{\partial t} =K_1\, _aD_x^{\alpha,\lambda } u(x,t)+K_2\,_xD_b^{\alpha,\lambda } u(x,t)+f(x,t),  &(x,t)\in (a,b)\times(0,T], \\\\
\displaystyle  u(x,0)=u_0(x),&x\in[a,b],\\\\
\displaystyle   u(a,t)=\phi_a(t),\,\,u(b,t)=\phi_b(t),& t\in[0,T],
\end{array}
\right.
\end{equation}
where $\lambda\geq0$.  %In this section, the mash in the section \ref{sec1} is still used.
%The diffusion coefficients $K_1$ and $K_2$
%are nonnegative constants and they satisfy $K_1^2+ K_2^2\neq0$.
% If $K_1\neq0$, then $\phi_a(t)\equiv0$ and $K_2\neq 0$, then $\phi_b(t)\equiv0$.  Next we discretize the equations (\ref{equ1}) by the proposed scheme.
% In the following  analysis of the numerical method, we suppose that (\ref{equ1})  has an unique and sufficiently smooth solution.
Utilizing the Crank-Nicolson technique for the time discretization of (\ref{equ1la}) and  fourth order quasi-compact discretization in space direction, we get
%\begin{equation}
%\begin{array}{lll}
%\displaystyle
%\frac{u_j^{n+1}-u_j^n}{\tau}=\frac{1}{2}\left(K_1(_aD_x^{\alpha,\lambda } u)_j^n + K_1(_aD_x^{\alpha,\lambda } u)_j^{n+1}  +K_2(_xD_b^{\alpha,\lambda } u)_j^n + K_2(_xD_b^{\alpha,\lambda } u)_j^{n+1} \right)%\\\\\displaystyle~~~~~~~~~~~~~~~~
%+f^{n+1/2}_j+O(\tau^2).
%\end{array}
%\end{equation}
%In  space discretization,   fourth order quasi-compact scheme  are used  to approximation the Riemann-Liouville fractional derivatives. This implies that
\begin{equation}
\begin{array}{l }
\displaystyle
P_x^\lambda\frac{u_j^{n+1}-u_j^n}{\tau}
=\frac{K_1\tau}{2 } \,_{L}D_{h }^{\alpha,\lambda } u_j^{n} +\frac{K_2\tau}{2 } \,_{R}D_{h }^{\alpha,\lambda } u_j^{n}+
\frac{K_1\tau}{2 } \,_{L}D_{h }^{\alpha,\lambda } u_j^{n+1} +\frac{K_2\tau}{2 } \,_{R}D_{h }^{\alpha,\lambda } u_j^{n+1}  \\\\
\displaystyle ~~~~  ~~~~  ~~~~  ~~~~  ~~~~   + P_x^\lambda f(x_j,t_{n+1/2})  + R^{n+1/2}_j,
\end{array}
\end{equation}
where
$$\,_{L}D_{h }^{\alpha,\lambda }  u_j^{n} =: \mu_1 \tilde{\Delta}_1^{\alpha,\lambda} u_j^{n}+ \mu_0 \tilde{\Delta}_0^{\alpha,\lambda} u_j^{n}+ \mu_{-1} \tilde{\Delta}_{-1}^{\alpha,\lambda} u_j^{n}=\frac{1}{h^\alpha}\sum_{k=0}^{j+1}w_{k}^{(\alpha,\lambda)}u_{j-k+1}^{n},$$
 $$\,_{R}D_{h }^{\alpha,\lambda} u_j^{n}=:\mu_1 \tilde{\Lambda}_1^{\alpha,\lambda} u_j^{n}+ \mu_0 \tilde{\Lambda}_0^{\alpha,\lambda} u_j^{n}+ \mu_{-1} \tilde{\Lambda}_{-1}^{\alpha,\lambda} u_j^{n}=\frac{1}{h^\alpha}\sum_{k=0}^{M-j+1}w_{k}^{(\alpha,\lambda)}u_{j+k-1}^{n},$$
%$P_x^\lambda u_{j,s}=b_2^\alpha e^{-\lambda h} u_{i-1,j}+(1-2b_2^\alpha) u_{j,s}  +b_2^\alpha e^{\lambda h} u_{i+1,j}$ ,
%When  third  order scheme () is used  to approximation the space derivative,    the coefficients $w_{3,k}^{(\alpha)}$ satisfy
%\begin{equation}\label{wl3}
%\sum_{k=0}^{i+1}w_{3,k}^{(\alpha)}u_{i-k+1}^{n}=\lambda_p\sum_{k=0}^{i+p}g_k^{(\alpha)}u_{i-k+p}^{n}+\lambda_q\sum_{k=0}^{i+q}g_k^{(\alpha)}u_{i-k+q}^{n},\quad (p,q)=(1,0),(1,-1)
%\end{equation}
the coefficients $w_0^{(\alpha,\lambda)}=\mu_1g_0^{(\alpha)}e^{\lambda h}$,  $w_1^{(\alpha,\lambda)}=\mu_1g_1^{(\alpha)}+\mu_0g_0^{(\alpha)}$, and $w_k^{(\alpha,\lambda)}=(\mu_1g_k^{(\alpha)}+\mu_0g_{k-1}^{(\alpha)}+\mu_{-1}g_{k-2}^{(\alpha)})e^{-(k-1)\lambda h}$, $k=2,\cdots, M$ and $R^{n+1/2}_j\leq C(\tau^2+h^4)$.
Denoting $U_j^n$ as the numerical approximation of $u_j^n$, we obtain the Crank-Nicolson quasi-compact scheme for (\ref{equ1la})
\begin{equation}\label{fourthccd4}
\begin{array}{lll}
\displaystyle
 P^\lambda_xU_j^{n+1}-
\frac{K_1\tau}{2 } \,_{L}D_{h}^{\alpha,\lambda } U_j^{n+1} -\frac{K_2\tau}{2 } \,_{R}D_{h}^{\alpha,\lambda } U_j^{n+1}\\\\
\displaystyle=P^\lambda_xU_{j,s}^{n}+\frac{K_1\tau}{2 } \,_{L}D_{h}^{\alpha,\lambda } U_j^{n} +\frac{K_2\tau}{2 } \,_{R}D_{h}^{\alpha,\lambda } U_j^{n}
+\tau P^\lambda_x f_j^{n+1/2}.
\end{array}
\end{equation}
For convenience, the approximation scheme (\ref{fourthccd4}) may be written in matrix form
\begin{equation}\label{matrixfourth4}
 (P^\lambda_\alpha-B^\lambda)U^{n+1}=(P^\lambda_\alpha+ B^\lambda)U^{n}+\tau F^n+H^\lambda,
 \end{equation}
 where $(P^\lambda_\alpha)_{j,s}=(P_\alpha)_{j,s}e^{(j-s)\lambda h}$, $B^\lambda=\frac{\tau}{2h^\alpha}(K_1A^\lambda_\alpha+K_2(A^\lambda_\alpha)^T)$,
$(A^\lambda_\alpha)_{j,s}=(A_\alpha)_{j,s}e^{(j-s)\lambda h}$,  $
U^n=(U^n_1,U^n_2,\cdots ,U^n_{M-1})^T$, and $F^n=(f^{n+1/2}_1,f^{n+1/2}_2,\cdots ,f^{n+1/2}_{M-1})^T
$. %is given by
\begin{rem}
Note that when taking $\lambda=0$, the tempered fractional diffusion equation (\ref{equ1la}) reduces to the fractional diffusion equation (\ref{equ1}) and its scheme (\ref{fourthccd4}) reduces to (\ref{fourthccd}).
\end{rem}

%---------------------------------------------------------------------------------------------------------------------------------------------------------------
\section{Numerical experiments}
For the numerical schemes of the fractional diffusion equation, we present some numerical results in one and two dimension cases to verify the theoretical results including the convergence orders and unconditional stability. For the tempered fractional diffusion equation, the numerical simulations are also performed which show the effectiveness of the proposed scheme; and the desired fourth order convergence is also obtained.
%Denote
\begin{example}%---------------------
Consider the following tempered space fractional diffusion equation
\begin{equation}\label{5.1}
\frac{\partial u}{\partial t}=\,_0D_x^{\alpha,\lambda }u(x)-e^{-t-\lambda x}\left(x^6+\frac{720 x^{6 -  \alpha}}{\Gamma(7 - \alpha)}\right) ,\quad (x,t)\in(0,1)\times(0,1],
\end{equation}
with the boundary conditions  $u(0,t)=0$ and  $u(1,t)=e^{-t-\lambda}$ and  the initial value  $u(x,0)=e^{-\lambda x}x^6, \,x\in[0,1]$. The exact solution is  $u(x)=e^{-t-\lambda x}x^{6}$.
\end{example}

In Table \ref{tablex}, we show that the quasi-compact scheme (\ref{fourthccd4}) is fourth order convergent in space.

\begin{table}[htbp]
\centering\small
\caption{  Numerical errors and convergence rates in   $L_2$ norm   to (\ref{5.1}) approximated by the quasi-compact difference scheme (\ref{fourthccd4}) at $t=1$ with $\tau=h^2$.}\vspace{1em} {\begin{tabular} {@{}cccccc@{}} \hline
 &  &\multicolumn{2}{c}{$\lambda=0$} & \multicolumn{2}{c}{$\lambda=1.5$ }\\
\cline{3-4} \cline{5-6}
$\alpha $&$M_X$ & $\|u-U\|_2$ &rate & $\|u-U\|_2$ &rate   \\
 \hline
1.1&$8$ & $   2.4321e-04   $ &   $         $   & $   1.6011e-04   $ &   $         $   \\
&$16 $&  $  1.3090e-05    $ &   $  4.2156    $   &  $  9.1799e-06   $ &   $   4.1245  $   \\
&$ 32$ &  $    7.4456e-07  $  &   $  4.1360    $     &  $    5.3102e-07  $  &   $  4.1117  $     \\
&$ 64$ &  $  4.4692e-08    $  &   $ 4.0583    $     &  $ 3.1555e-08   $  &   $ 4.0728   $     \\
&$128$ &  $   2.7455e-09   $  &   $    4.0249    $   &  $   1.9171e-09   $  &   $     4.0408    $    \\
 \hline
1.5&$8$ & $ 1.2806e-04      $ &   $          $   & $  7.5690e-05   $ &   $          $   \\
&$16 $&  $   8.0137e-06 $ &   $   3.9982   $      &  $    4.7019e-06 $ &   $    4.0088 $      \\
&$ 32$ &  $  5.0273e-07    $  &  $    3.9946   $      &  $  2.9387e-07    $  &  $    4.0000 $     \\
&$ 64$ &  $   3.1507e-08    $  &   $  3.9960   $      &  $  1.8396e-08   $  &   $   3.9978  $     \\
&$128$ &  $  1.9724e-09      $  &   $    3.9976   $   &  $   1.1512e-09     $  &   $     3.9981$   \\
\hline
1.9&$8$ & $  4.4604e-05    $ &   $         $ &     $   2.3601e-05   $ &   $         $   \\
&$16 $&  $   2.8032e-06    $ &   $  3.9920   $    &  $   1.4844e-06    $ &   $  3.9909  $    \\
&$ 32$ &  $  1.7561e-07    $  &  $   3.9967   $     &  $    9.2998e-08 $  &  $    3.9965 $     \\
&$ 64$ &  $    1.0987e-08   $  &  $     3.9985  $     &  $     5.8188e-09  $  &  $       3.9984  $     \\
&$128$ &  $   6.8700e-10    $  &  $  3.9993    $  &  $    3.6385e-10   $  &  $   3.9993   $   \\
\hline
\end{tabular}}\label{tablex}
\end{table}

\begin{example}%---------------------
Consider the following space fractional diffusion equation%two dimensional
\begin{equation}\label{5.2}
\frac{\partial u}{\partial t}=\,_0D_x^{\alpha } u(x)+\,_xD_1^{\alpha } u(x)+f(x,t) ,\quad  (x,t)\in(0,1)\times(0,1].
\end{equation}
Then the source term is
$$
\begin{array}{lll}
\displaystyle
f(x,t)=-e^{-t}(x^5(1-x)^5-\Gamma(11)( x^{10 - \alpha}+(1-x)^{10 - \alpha}) /\Gamma(11 -\alpha)\\\\
\displaystyle~~  ~~~~ ~~~~   + 5 \Gamma(10)( x^{9 - \alpha}+ (1-x)^{9 - \alpha} )/\Gamma(10 - \alpha)%\\\\\displaystyle~~~~ ~~~~ ~~~~~~~~
-10 \Gamma(9) (x^{8 - \alpha}+ (1-x)^{8 - \alpha}) /\Gamma(9 - \alpha)\\\\
\displaystyle~~    ~~~~~~~~ + 10 \Gamma(8)( x^{7 - \alpha}+(1-x)^{7 - \alpha} )/\Gamma(8 - \alpha)%\\\\\displaystyle~~~~ ~~~~ ~~~~~~~~
- 5 \Gamma(7) (x^{6 - \alpha}+(1-x)^{6 - \alpha}) /\Gamma(7 - \alpha)\\\\
\displaystyle~~   ~~~~~~~~ +\Gamma(6)(x^{5 - \alpha}+(1-x)^{5 - \alpha})/\Gamma(6 - \alpha) .
\end{array}
$$
The exact solution is given by  $u(x)=e^{-t}x^{5}(1-x)^5$. In the domain  $t\in[0,1]$,  the boundary conditions are  $u(0,t)=0$ and $u(1,t)=0$.  The initial value is  $u(x,0)=x^5(1-x)^5,\, x\in[0,1]$.
\end{example}

 Table \ref{tabex12}  shows that the quasi-compact scheme (\ref{fourthccd}) to solve the one dimensional two sided fractional
 diffusion equation also is fourth order convergent.
%\begin{table}[htbp]
%\centering
%\caption{在$t=1$时刻当取$\tau=h^2$时，提出的三阶C-N结合的紧格式的最大误差范数和$L_2$范数以及他们的收敛阶} {\begin{tabular}{@{}cccccc@{}} \hline
% $\alpha$ & $N$ & $\|u-U\|_2$ &rate & $\|u-U\|_\infty$ &rate  \\
% \hline
%1.1&$8$ & $   *       $ &   $         $ & $  * $ &   $ $  \\
%&$16 $&  $    *  $ &   $   3.6129 $  &  $     * $ &   $    3.5358 $  \\
%&$ 32$ &  $   * $  &   $    3.7753  $   & $    *   $  &    $      3.7665  $  \\
%&$ 64$ &  $    * $  &   $   3.8821 $   &  $     * $  &    $     3.8809   $  \\
%&$128$ &  $    *  $  &   $    3.9389 $  &  $    *  $  &   $       3.9392$ \\
% \hline
%1.5&$8$ & $     *   $ &   $          $ & $     *  $  &  $  $   \\
%&$16 $&  $   * $ &   $ 3.8135  $    &  $      * $ &   $      3.8915  $  \\
%&$ 32$ &  $    * $  &  $   3.8720   $    &  $        *   $  &   $    3.9098  $  \\
%&$ 64$ &  $    *   $  &   $  3.9326  $   &  $      * $  &    $      3.9615  $  \\
%&$128$ &  $       * $  &   $   3.9650  $   &  $      * $  &   $     3.9459 $ \\
%\hline
%1.9&$8$ & $   *$ &   $         $ & $     *  $ &   $      $  \\
%&$16 $&  $    *  $ &   $    4.1437 $  &  $       *  $ &   $       4.2329 $  \\
%&$ 32$ &  $      * $  &  $   3.9969   $   &  $        * $  &    $     4.0622 $  \\
%&$ 64$ &  $    * $  &  $    3.9925  $   &  $     * $  &    $     4.0283$  \\
%&$128$ &  $    *  $  &  $     3.9955$  &  $    * $  &   $  3.9761   $ \\
%\hline
%\end{tabular}}
%\end{table}

\begin{table}[htbp]
\centering\small
\caption{Numerical errors and  convergence rates in $L_\infty$ norm and $L_2$ norm   to (\ref{5.2}) approximated by the quasi-compact difference scheme (\ref{fourthccd}) at $t=1$ with $\tau=h^2$.} \vspace{1em}  {\begin{tabular} {@{}cccccccccc@{}} \hline
$\alpha $&$M_x$ & $\|u-U\|_2$ &rate & $\|u-U\|_\infty$ &rate\\
 \hline
1.1&$8$ & $    9.4394e-07          $ &   $         $ & $  1.4488e-06  $ &   $ $   \\
&$16 $&  $    7.7153e-08  $ &   $   3.6129 $  &  $     1.2492e-07 $ &   $    3.5358 $   \\
&$ 32$ &  $   5.6349e-09 $  &   $    3.7753  $   & $    9.1789e-09   $  &    $      3.7665  $  \\
&$ 64$ &  $    3.8217e-10 $  &   $   3.8821 $   &  $     6.2304e-10 $  &    $     3.8809   $   \\
&$128$ &  $    2.4920e-11  $  &   $    3.9389 $  &  $     4.0617e-11$  &   $       3.9392$  \\
 \hline
1.5&$8$ & $     1.4931e-06   $ &   $          $ & $     2.5326e-06  $ &$ $  \\
&$16 $&  $   1.0619e-07 $ &   $ 3.8135  $    &  $      1.7066e-07 $ &   $      3.8915  $   \\
&$ 32$ &  $    7.2530e-09 $  &  $   3.8720   $    &  $        1.1354e-08$  &   $    3.9098  $  \\
&$ 64$ &  $     4.7498e-10   $  &   $  3.9326  $   &  $      7.2882e-10 $  &    $      3.9615  $   \\
&$128$ &  $       3.0416e-11 $  &   $   3.9650  $   &  $       4.7293e-11 $  &   $     3.9459 $   \\
\hline
1.9&$8$ & $   1.5101e-06$ &   $         $ & $     2.6288e-06  $ &   $      $   \\
&$16 $&  $    8.5433e-08  $ &   $    4.1437 $  &  $       1.3980e-07  $ &   $       4.2329 $   \\
&$ 32$ &  $      5.3511e-09 $  &  $   3.9969   $   &  $        8.3686e-09 $  &    $     4.0622 $   \\
&$ 64$ &  $    3.3620e-10 $  &  $    3.9925  $   &  $     5.1288e-10 $  &    $     4.0283$   \\
&$128$ &  $    2.1078e-11  $  &  $     3.9955$  &  $    3.2590e-11 $  &   $  3.9761   $  \\
\hline
\end{tabular}}\label{tabex12}
\end{table}

\begin{example}%---------------------
The following two dimensional two sided fractional diffusion problem
\begin{equation}\label{5.3}
\frac{\partial u(x,y,t)}{\partial t}=\,_0D_x^{\alpha  }u(x,y,t)+\,_xD_1^{\alpha  }u(x,y,t)+\,_0D_y^{\beta  }u(x,y,t)+\,_yD_1^{\beta  } u(x,y,t)+f(x,y,t) ,
\end{equation}
is considered   in the domain $ \Omega= (0,1)^2$ and $t\in(0,1]$.
The source term is
$$
\begin{array}{lll}
\displaystyle
f(x,t)=-10^6e^{-t}\left[x^5(1-x)^5y^5(1-y)^5\right.\\\\
\displaystyle~~~~     ~~~~~~~~-\left(\frac{\Gamma(11)}{\Gamma(11 -\alpha)}( x^{10 - \alpha}+(1-x)^{10 - \alpha}) \right.
+ \frac{5 \Gamma(10)}{\Gamma(10 - \alpha)}( x^{9 - \alpha}+ (1-x)^{9 - \alpha} )\\\\
\displaystyle     ~~~~~~~~~~~~ %\\\\\displaystyle~~~~ ~~~~ ~~~~~~~~
-\frac{10 \Gamma(9)}{\Gamma(9 - \alpha)} (x^{8 - \alpha}+ (1-x)^{8 - \alpha}) + \frac{10 \Gamma(8)}{\Gamma(8 - \alpha)}( x^{7 - \alpha}+(1-x)^{7 - \alpha} )
%\displaystyle~~~~ ~~~~ ~~~~~~~~%\\\\\displaystyle~~~~ ~~~~ ~~~~~~~~\\\\
\\\\
\displaystyle     ~~~~~~~~~~~~- \frac{5 \Gamma(7)}{\Gamma(7 - \alpha)} (x^{6 - \alpha}+(1-x)^{6 - \alpha})  \left.+\frac{\Gamma(6)}{\Gamma(6 - \alpha)}(x^{5 - \alpha}+(1-x)^{5 - \alpha})\right)y^5(1-y)^5 \\\\
\displaystyle ~~~~~~~~~~~~ -\left(\frac{\Gamma(11)}{\Gamma(11 -\beta)}( y^{10 - \beta}+(1-y)^{10 - \beta}) \right.+ \frac{5 \Gamma(10)}{\Gamma(10 - \beta)}( y^{9 -\beta}+ (1-y)^{9 - \beta} )\\\\
\displaystyle  ~~~~~~~~ ~~~~  %\\\\\displaystyle~~~~ ~~~~ ~~~~~~~~
-\frac{10 \Gamma(9) }{\Gamma(9 - \beta)}(y^{8 - \beta}+ (1-y)^{8 - \beta}) + \frac{10 \Gamma(8)}{\Gamma(8 - \beta)}( x^{7 - \beta}+(1-x)^{7 - \beta} )\\\\
%\displaystyle~~~~ ~~~~    ~~~~ ~~~~~~~~ ~~~~~~~~%\\\\\displaystyle~~~~ ~~~~ ~~~~~~~~
%\\\\
\displaystyle  ~~~~~~~~  ~~~~\left.- \frac{5 \Gamma(7)}{\Gamma(7 - \beta)} (y^{6 - \beta}+(1-y)^{6 - \beta}) \left. +\frac{\Gamma(6)}{\Gamma(6 - \beta)}(y^{5 - \beta}+(1-y)^{5 - \beta})\right)x^5(1-x)^5 \right].
\end{array}
$$
The exact solution is given by  $u(x)=10^6e^{-t}x^{5}(1-x)^5 y^{5}(1-y)^5$. The boundary condition is $u(x,y,t)=0$  with $(x,y)\in\partial \Omega$ and $t\in[0,1]$.  The initial value is  $u(x,y,0)=10^6x^5(1-x)^5 y^5(1-y)^5$ with $ (x,y)\in[0,1]^2$.
\end{example}

In Table \ref{t6.3}, we present the numerical errors $\|u-U\|_2$  and the corresponding convergence orders with space  step size $h_x=h_y$, where
$U$ is the solution of the quasi-compact difference scheme (\ref{d1adi}) or (\ref{d2adi}). It can be noted that the schemes are fourth order convergent, which is in agreement with the theoretical convergence analysis.

\begin{table}[htbp]
\centering\small
%\resizebox{\textwidth}{!}{%
\caption{Numerical errors and  convergence rates in $L_2$ norm   to (\ref{5.3}) approximated by the quasi-compact difference schemes (\ref{d1adi}) and (\ref{d2adi}), respectively, at $t=1$ with $\tau=h_x^2=h_y^2$.}\vspace{1em} {\begin{tabular} {@{}cccccc @{}} \hline
 $ $ & $$ &  \multicolumn{2}{c}{$(\alpha,\beta)=(1.1,1.5)$} & \multicolumn{2}{c}{$(\alpha,\beta)=(1.4,1.9)$}  \\
\cline{3-4} \cline{5-6}
&$M_x$ & $\|u-U\|_2$ &rate  & $\|u-U\|_2$ &rate   \\
 \hline
D'yakonov&$8$ & $  7.2903e-04      $ &   $         $  & $   8.4729e-04     $ &   $         $   \\
&$16 $&  $    5.3915e-05$ &   $    3.7572 $   &  $  5.7210e-05 $ &   $    3.8885  $    \\
&$ 32$ &  $  3.7385e-06$  &   $     3.8502 $     &  $  3.8200e-06$  &   $   3.9046  $    \\
&$ 64$ &  $    2.4685e-07$  &   $   3.9207 $    &  $  2.4748e-07$  &   $3.9482 $      \\
&$128$ &  $    1.5880e-08 $  &   $   3.9584  $    &  $   1.5763e-08  $  &   $    3.9727  $   \\
 \hline
Douglas&$8$ & $    7.2903e-04 $ &   $          $  & $   8.4729e-04  $ &   $          $   \\
&$16 $&  $   5.3915e-05 $ &   $   3.7572  $      &  $  5.7210e-05 $ &   $ 3.8885  $      \\
&$ 32$ &  $     3.7385e-06$  &  $    3.8502    $      &  $  3.8200e-06 $  &  $   3.9046   $     \\
&$ 64$ &  $     2.4685e-07 $  &   $  3.9207  $     &  $  2.4748e-07 $  &   $   3.9482   $     \\
&$128$ &  $     1.5880e-08$  &   $  3.9584  $    &  $  1.5763e-08$  &   $   3.9727  $     \\
%\hline
%1.9&$8$ & $    7.5730e-06$ &   $         $ & $      1.0729e-05 $ &   $      $  \\
%&$16 $&  $      1.6249e-06 $ &   $     2.2205 $  &  $        2.3574e-06 $ &   $       2.1862 $  \\
%&$ 32$ &  $     3.9976e-07 $  &  $     2.0231  $   &  $        5.7644e-07 $  &    $     2.0319$  \\
%&$ 64$ &  $     9.9811e-08$  &  $     2.0019  $   &  $       1.4365e-07$  &    $     2.0046$  \\
%&$128$ &  $    2.4954e-08 $  &  $     2.0000$  &  $    3.5939e-08 $  &   $    1.9990 $ \\
\hline
\end{tabular}}\label{t6.3}
\end{table}

\section{Conclusions}
The continuous time random walk (CTRW) model is the basic stochastic process in statistical physics. The CTRW model characterizes the L{\'e}vy flight if the first moment of the distribution of the waiting time is finite, and the jump length obeys the power law distribution and its second moment is infinite; the corresponding Fokker-Planck equation of the process is the space fractional diffusion equation. Sometimes because of the limit of space size, the power law distribution of the jump length has to be tempered. The Fokker-Planck equation of the new stochastic process is the tempered space fractional diffusion equation. This paper provides the basic strategy of deriving the quasi-compact high order discretizations for space fractional derivative and tempered space fractional derivative. As concrete examples, fourth order discretizations are detailedly discussed and applied to solve the (tempered) space fractional diffusion equation, and the extensive numerical simulations confirm the effectiveness of the provided schemes. In fact, the strict numerical stability and convergence analysis are also performed for the one and two dimensional space fractional diffusion equations.

%Based on the operator splitting techniques, we introduce the unconditional stable semi-implicit numerical schemes for subdiffusive reaction diffusion equation with Dirichlet boundary condition and Neummann boundary condition, respectively. And the subdiffusive predator-prey model is detailedly discussed. We prove that its analytical solution is positive and bounded. Then we show that the proposed numerical schemes preserve the positivity and boundedness of the analytical solutions. The extensive numerical experiments are performed to confirm the theoretical results and show the dissipative properties of subdiffusive predator-prey model.
%

%\section*{Acknowledgements}
%This work was supported by the National Natural Science Foundation of China under Grant  No. 11471150 and No. 11271173.

% ----------------------------------------------------------------
\bibliographystyle{elsarticle-num}
\bibliography{<your-bib-database>}
\end{document}